\documentclass{article}
\usepackage{amsmath}
\usepackage{amssymb}
\usepackage{amsthm}
\usepackage{bbm,mathtools}
\usepackage{enumerate}
\usepackage{graphicx}
\usepackage{color}
\usepackage{floatrow}
\usepackage{caption}
\usepackage[symbol]{footmisc}

\renewcommand{\thefootnote}{\fnsymbol{footnote}}
\usepackage{amsmath}
\usepackage{mathabx}

\newcommand{\one}{\mathbbm 1}

\def\comp{\raise 1pt \hbox{$\scriptstyle\bullet$}}

\def\upto{{\raise 1pt \hbox{$\scriptstyle \,\nearrow\,$}}}
\def\downto{{\raise 1pt \hbox{$\scriptstyle \,\searrow\,$}}}

\def\cl{\mathop{\rm cl}\nolimits}

\def\cl{\mathop{\rm cl}\nolimits}

\def\A{{\cal A}}
\def\B{{\cal B}}

\def\E{{\cal E}}
\def\F{{\cal F}}

\def\G{{\cal G}}
\def\H{{\cal H}}

\def\J{{\cal J}}
\def\K{{\cal K}}

\def\M{{\mathbb {M}}}

\def\P{{\mathbb {P}}}

\def\S{{\cal S}}
\def\Q{{\cal Q}}

\def\V{{\cal V}}

\def\X{{\cal X}}

\newtheorem{theorem}{Theorem}[section]
\newtheorem{lemma}[theorem]{Lemma}
\newtheorem{corollary}[theorem]{Corollary}
\newtheorem{proposition}[theorem]{Proposition}
\newtheorem{example}[theorem]{Example}
\newtheorem{remark}[theorem]{Remark}
\newtheorem{definition}[theorem]{Definition}
\newtheorem{assumption}[theorem]{Assumption}
\newtheorem{property}[theorem]{Property}

\begin{document}
\title{The Riesz representation theorem and weak${}^*$ compactness of semimartingales\footnote{Research is partly supported by a Swiss National Foundation Grant SNF 200020-172815.}}
\author{Matti Kiiski\footnote{ETH Z\"urich, Department of Mathematics, R\"amistrasse 101, 8092, Z\"urich, Switzerland.}}
\maketitle

\begin{abstract}
We show that the sequential closure of a family of probability measures on the canonical space of c{\`a}dl{\`a}g paths satisfying Stricker's uniform tightness condition is a weak${}^*$ compact set of semimartingale measures in the pairing of the Riesz representation theorem under topological assumptions on the path space. Similar results are obtained for quasi- and supermartingales under analogous conditions. In particular, we give a full characterization of the strongest topology on the Skorokhod space for which these results are true.
\end{abstract}

\renewcommand*{\thefootnote}{\arabic{footnote}}

\noindent\textbf{Keywords and phrases.} Skorokhod space, Meyer-Zheng topology, $S$-topology, weak${}^*$ topology,  c{\`a}dl{\`a}g semimartingale
\newline
\newline
\noindent\textbf{AMS subject classification codes.} 28C05, 54D30, 60B05, 60G05

\section{Introduction}
\label{sec:intro}

The Riesz representation theorem states that the operation of integration defines a one-to-one correspondence between the continuous linear functionals on the bounded continuous functions and the Radon measures on a topological space. On the Skorokhod space, it provides a locally convex way of constructing all c{\`a}dl{\`a}g stochastic processes on the canonical space as \emph{tight} probability measures. On conceptual level, any criteria that characterizes a certain object should give rise to some kind of compactness when applied uniformly to a family of objects. We relate Stricker's uniform tightness condition of semimartingales to the weak${}^*$ compactness in the pairing of the Riesz representation theorem on the canonical space of c{\`a}dl{\`a}g paths.

The weak topologies on the Skorohod space and weak convergence of stochastic processes, that is, the sequential convergence of the weak${}^*$ topology of the Riesz representation theorem, has been earlier studied in the works of Meyer and Zheng \cite{meyerzheng},
Zheng \cite{zheng}, Stricker \cite{stricker}, Jakubowski, M{\' e}min and Pages \cite{jakubowski6}, Kurtz \cite{kurtz}, Lowther \cite{lowther2009} and Jakubowski \cite{jakubowski}, \cite{jakubowski4}. We rely heavily on these earlier results. In particular, we utilize the stability results of Kellerer \cite{kellerer}, Meyer and Zheng \cite{meyerzheng}, Jakubowski, M{\' e}min and Pages \cite{jakubowski6}, Jakubowski \cite{jakubowski}, and Lowther \cite{lowther2009}. 

Weak topologies are rich in terms of convergent subsequences and have found various applications in studying convergence of financial markets Prigent \cite{prigent2003}, time series analysis in econometrics Chan and Zhang \cite{chan2010}, stochastic optimal control Kurtz and Stockbridge \cite{kurtz2001}, Bahlali and Gherbal \cite{bahlali2011}, Tan and Touzi \cite{tan2013}, and martingale optimal transport, as introduced by Beiglb{\"o}ck, Henry-Labord{\`e}re
and Penkner \cite{beiglbock}, and extended for continuous time parameter by Dolinsky and Soner \cite{mete2}, and Guo, Tan and Touzi \cite{touzi}. We aim to provide a functional analytic framework that unifies and elaborates these existing results and allows to extend the analysis beyond the convergence of sequences. In particular, the framework allows to study the non-sequential compactness of families of semimartingales. The question of compactness arises naturally in the context of convex conjugate duality on functions and measures on the c{\`a}dl{\`a}g path space. Many problems in quantitative risk management can be embedded in this framework as convex risk measures and their conjugates F{\" o}llmer and Schied \cite{follmer}. The classical example is the minimal superhedging cost of a derivative contract over a convex set hedging positions. We recall the connection between the compactness and the superhedging duality that yields model-independent price bounds for derivative contracts as originally observed by Beiglb{\"o}ck, Henry-Labord{\`e}re and Penkner \cite{beiglbock}, and extended to the continuous time by Dolinsky and Soner \cite{mete2}, and Guo, Tan and Touzi \cite{touzi}.

The objective of the paper is to provide a weak${}^*$ compactness result for c{\`a}dl{\`a}g semimartingales under the most general topological assumption on the path space. Our main contribution is to unify the previous results on the weak convergence of semimartingales and provide an easy method for constructing weak${}^*$ compact sets of semimartingales on the canonical space of c{\`a}dl{\`a}g paths. We also give examples of such sets and show that the examples are consistent with earlier results for Banach spaces of stochastic processes defined over a common probability space.


We characterize the strongest topology on the Skorokhod space for which our main result is true. A natural candidate is Jakubowski's $S$-topology, due to its tightness criteria. However, it is an open problem whether the $S$-topology is sufficiently regular. We address the problem of regularity by introducing a new weak topology on the Skorokhod space that has the same continuous functions as the $S$-topology, suitable compact sets and additionally satisfies a strong separation axiom. The topological space is perfectly normal ($T_6$) in comparison to the Hausdorff property ($T_2$) that has been verified for the $S$-topology. The topology is obtained from Jakubowski's $S$-topology as a result of a standard regularization method that appears already in the classical works of Alexandroff \cite{alexandroff} and Knowles \cite{knowles}. Our contribution is to carefully show that the important properties of the $S$-topology are preserved in the regularization.

The rest of the paper is organized as follows.

In Section~\ref{sec:preli}, we give rigorous definitions of semimartingale measures and related notions on the canonical space of c{\`a}dl{\`a}g paths. We also provide a brief introduction to the aforementioned Riesz representation theorem that is the basis of our approach. The main results, examples and financial motivation are given in Section~\ref{sec:mainres}. The proofs of the main results and the required auxiliary results are provided in Section~\ref{sec:basicresults}. In Section~\ref{snot}, we characterize the strongest topology on the Skorokhod space for which the results of the two previous sections are true. Some definitions and technical results are omitted in the main part of the article and are gathered in Appendix~\ref{sec:appendix}.

\subsubsection*{Conventions and notations}
\label{sec:not}

Throughout, the comparatives 'weaker' and 'stronger' should be understood in the wide sense 'weaker or equally strong' and 'stronger or equally strong', respectively. We say that two topologies are 'comparable', if one is stronger than another.

We fix the following notations.

$\mathbb{N}$ denotes the family of natural numbers and $\mathbb{N}_0:=\{0\}\cup\mathbb{N}$.

$\mathbb{R}$ (resp. $\mathbb{R}_{+}$) denotes the family of real (resp. non-negative real) numbers and $\overline{\mathbb{R}}:=\mathbb{R}\cup\{\pm\infty\}$.

$\mathbb{Q}$ denotes the family of rational numbers.

$|x|:=|x_1|+|x_2|+\cdots+|x_d|$, $x=(x_1,x_2,\dots,x_d)\in\mathbb{R}^d$.

$x\vee y:=\max\{x,y\}$, $x\wedge y:=\min\{x,y\}$, $x^+:=x\vee 0$, $x^-:=-(x\wedge 0)$, $x,y\in\mathbb{R}$.

$x^+=x^+_1+x^+_2+\cdots+x^+_d$, $x^-=x^-_1+x^-_2+\cdots+x^-_d$, $x=(x_1,x_2,\dots,x_d)\in\mathbb{R}^d$.

$\|x\|_\infty:=|x_1|\vee|x_2|\vee\cdots\vee|x_d|$, $x=(x_1,x_2,\dots,x_d)\in\mathbb{R}^d$.

$g\circ f$ denotes the composition mapping of $f:X\rightarrow Y$ and $g: Y\rightarrow Z$ i.e. $g\circ f: X\rightarrow Z$, where $[g\circ f](x):=g(f(x))\in Z$, $x\in X$.

$\mathbb{D}$, $\mathbb{D}(I)$ and $\mathbb{D}(I;\mathbb{R}^d)$ denote the $\mathbb{R}^d$-valued c{\`a}dl{\`a}g functions on $I$.

$\mathbb{V}(I)$ denotes the functions of finite variation on $I$.

$\mathbb{C}(\mathbb{X})$ (resp. $\mathbb{C}_{b}(\mathbb{X})$) denotes the family of continuous (resp. bounded continuous) functions on a topological space $\mathbb{X}$; e.g., $\mathbb{X}=\mathbb{D}$.

$\mathbb{U}(\mathbb{D})$ (resp. $\mathbb{U}_{b}(\mathbb{D})$) denotes the family of upper semi-continuous (resp. bounded~upper semi-continuous) functions on $\mathbb{D}$.

$\mathbb{B}_{b}(\mathbb{D})$ (resp. $\mathbb{B}_{b}(\mathbb{D})$) denotes the family of Borel (resp. bounded Borel) functions on $\mathbb{D}$.

$\mathbb{B}_0(\mathbb{D})$ denotes the family of bounded Borel functions vanishing at infinity on $\mathbb{D}$.

$\M_t(\mathbb{D})$ denotes the family of Radon measures of finite variation on the Skorokhod space $\mathbb{D}$.

$\M_\tau(\mathbb{D})$ denotes the family of $\tau$-additive Borel measures of finite variation on the Skorokhod space $\mathbb{D}$.

$\M_\sigma(\mathbb{D})$ denotes the family of $\sigma$-additive Borel measures of finite variation on the Skorokhod space $\mathbb{D}$.

If $\M_t(\mathbb{D})=\M_\tau(\mathbb{D})=\M_\sigma(\mathbb{D})$, then we denote the all three families by $\M(\mathbb{D})$.

$\M_+(\mathbb{D})$ denotes the family of non-negative elements of $\M(\mathbb{D})$.

$\P(\mathbb{D})$ (resp. $\P(\mathbb{R^d})$) denotes the family of all probability measures on the Skorokhod space $\mathbb{D}$ (resp. the Euclidean space $\mathbb{R}^d$).



$\K(\mathbb{D})$, $\B(\mathbb{D})$ and $\B a(\mathbb{D})$ denote the family of compact, Borel and Baire sets on the Skorokhod space $\mathbb{D}$, respectively. Similarly, $\B(\mathbb{R}^d)$ denotes the Borel sets on the Euclidean space $\mathbb{R}^d$.

For a process $X:\Omega\times I\rightarrow\mathbb{R}^d$ and a subset $J\subset I$, we write $X_J$ for the restriction of $X$ on $J$ i.e. $X_J:\Omega\times J\rightarrow\mathbb{R}^d$. In the case of a singleton $J=\{t\}$, $t\in I$, we suppress the dependence on the set $J$ and simply write $X_t$ for the value of $X$ at $t$, whence $X_t:\Omega\rightarrow\mathbb{R}$.

$\|X\|_{\infty}(\omega):=\sup_{t\in I}\|X_t(\omega)\|_\infty$, for $X:\Omega\times I\rightarrow\mathbb{R}^d$.


$\|X\|_\mathbb{V}(\omega):=\sum_{i=1}^d\sup_{\pi}\{|X^i_0(\omega)|+\sum^n_{k=1}|X^i_{t_k}(\omega)-X^i_{t_{k-1}}(\omega)|\}$, for $X:\Omega\times I\rightarrow\mathbb{R}^d$, where the supremum is taken over all finite partitions $\pi$ of $I$.

$E_Q[f]:=\int fdQ$ denotes the integral, for a measurable function $f:\Omega\rightarrow\mathbb{R}$ and a probability measure $Q$ on $(\Omega,\mathcal{F})$.

$\|f\|_{L^p(Q)}:=\left(E_Q[|f|^p]\right)^{1/p}$, $p\geq 1$, denotes the $L^p$-norm, for a measurable function $f:\Omega\rightarrow\mathbb{R}^d$ and a probability measure $Q$ on $(\Omega,\mathcal{F})$.

$\|X\|_{L^{p,\infty}(Q)}:=\|\|X\|_\infty\|_{L^p(Q)}$, for $X:\Omega\times I\rightarrow\mathbb{R}^d$ and a probability measure $Q$ on $(\Omega,\mathcal{F})$.


$\|X\|_{\H^p(Q)}:=\|\|M\|_\infty+\|A\|_\mathbb{V}\|_{L^p(Q)}$, $p\geq 1$, is the (maximal) $\H^p$-norm, for a semimartingale $X$ whose canonical decomposition under $Q$ is $X=M+A$.

$\E(Q)$ denotes the family of elementary predictable processes bounded by $1$, for a probability measure $Q$ on $(\Omega,\F)$.

$(H\bullet X):=H_0X_0+\int Hd X$ denotes the stochastic integral, for a probability measure $Q$ on $(\Omega,\F)$; the dependence on $Q$ is omitted in this notation.

$\|X\|_{\E^p(Q)}:=\sup_{H\in\E(Q)}\|(H\bullet X)\|_{L^p(Q)}$, $p\geq 1$, for some probability measure $Q$ on $(\Omega,\F)$.



$N^{a,b}_\pi$ denotes the number of upcrossings of an interval $[a,b]$ w.r.t. $\pi$.

$N^{a,b}=\sup_\pi N^{a,b}_\pi$ denotes the number of upcrossings of an interval $[a,b]$.

$\Delta X:=X_t-X_{t-}$ denotes the jump of $X$ at $t$.

$[\omega]^t$ denotes the restriction of $\omega$ on $[0,t]$.

\subsubsection*{Terminology}
\label{sec:terminology}

We provide some frequently used terminology. Standard literature references for general topology and topological measure theory are \cite{engelking} and \cite{bogachev}.

Let $X$ be a topological space. A subset $Z\subset X$ is called:

\emph{Compact}, if every cover of this set by open sets contains a finite subcover.

\emph{Relatively compact}, if this set is contained in a compact
set.

\emph{Sequentially compact}, if every infinite sequence of its
elements contains a subsequence converging to an element of $Z$.

\emph{Relatively sequentially compact}, if every infinite sequence
of its elements contains a subsequence converging in $X$.

\begin{remark}
In contrast to the case of a metric space, in a general topological space, neither does compactness imply sequential compactness nor the other way around.
\end{remark}

The \emph{closure} $\cl Z$ of $Z$ is the set of all points $x\in X$ such that every neighborhood of $x$ contains at least one point of $Z$.

The \emph{sequential closure} $[Z]_{seq}$ of $Z$ is the set of all points $x\in X$ for which there is a sequence in $Z$ that converges to $x$.

\begin{remark}
In contrast to the case of a metric space, in a general topological space, the sequential closure of a set is not necessarily a sequentially closed set.
\end{remark}

All topological spaces considered are \emph{Hausdorff} ($T_2$) and a Hausdorff space $X$ is called:

\emph{Regular} ($T_{3}$), if for every point $x\in X$ and every closed set $Z$ in $ X$ not containing $x$, there exists disjoint open sets $U$ and $V$ such that $x\in U$ and $Z\subset V$.

\emph{Completely regular} ($T_{3{^1/_2}}$), if for every point $x\in X$ and every closed set $Z$ in $X$ not containing $x$, there exists a continuous function $f:X\rightarrow [0,1]$ such that $f(x)=1$ and $f(z)=0$ for all $z\in Z$.

\emph{Perfectly normal} ($T_6$), if every closed set $Z\subset X$ has the form $Z=f^{-1}(0)$ for some continuous function $f$ on $ X$.

\emph{Paracompact}, if every open cover of $X$ has an open refinement that is locally finite.

\emph{$k$-space}, if the set $Z\subset X$ is closed in $X$ provided that the intersection of $Z$ with any compact subspace $K$ of the space $X$ is closed in $K$.

\emph{Sequential space}, if every sequentially closed set is closed.

\emph{Fr{\' e}chet-Urysohn space}, if every subspace is a sequential space.

\emph{Polish space}, if the space is homeomorphic to a complete separable metric space.

\emph{Lusin space}, if the space is the image of a complete separable metric space under a continuous one-to-one mapping.

\emph{Souslin space}, if the space is the image of a complete separable metric space under a continuous mapping.

\emph{Radon space}, if every Borel measure on the space is a Radon measure.

\emph{Perfect space}, if every Borel measure on the space is perfect i.e. for every Borel measurable function $f$ for every Borel measure $Q$ the set $f(X)$ contains a Borel set $B$ for which $Q[f^{-1}(B)]=f(X)$

\emph{Angelic space}, if every set $Z\subset X$ with the property that every infinite sequence of its elements has a limit point in $X$, possesses also the following properties: $Z$ is relatively compact and each point in the closure of $Z$ is the limit of some sequence in $Z$.

\begin{remark}
\label{rem:her}
For closed subspaces, all these properties are hereditary, meaning that, if the space has the property, then a closed subspace endowed with the relative topology has the property as well. So, all discussion on these properties generalizes as such for relative topologies on closed sets.
\end{remark}

\section{C{\`a}dl{\`a}g semimartingales as linear functionals}
\label{sec:preli}

In this preliminary section, we define the canonical space for c{\`a}dl{\`a}g semimartingales, and related measures and continuous linear functionals.

\subsection{Canonical space of c{\`a}dl{\`a}g paths}

We fix $I$ to denote a usual time index set of a stochastic process, i.e., $I:=[0,T]$ for $0<T\leq\infty$ or $I:=[0,\infty)$. The Skorokhod space $\mathbb{D}(I;\mathbb{R}^d)$, $d\in\mathbb{N}$, with the domain $I$ consists of $\mathbb{R}^d$-valued c{\`a}dl{\`a}g functions $\omega$ on $I$ that admit a limit $\omega(t-)$ from left, for every $t>0$, and are continuous from right, $\omega(t)=\omega(t+)$, for every $t<T$. The space $\mathbb{D}([0,\infty];\mathbb{R}^d)$ is regarded as a product space $\mathbb{D}([0,\infty);\mathbb{R}^d)\times\mathbb{R}^d$; see Appendix~\ref{sec:othertopos}. We write $\omega=(\omega^1,\dots,\omega^d)$, for $\omega\in\mathbb{D}(I;\mathbb{R}^d)$, if $\omega(t)=(\omega^1(t),\dots,\omega^d(t))$, for every $t\in I$. We denote by $X$ the canonical process of $\mathbb{D}(I;\mathbb{R}^d)$, i.e., $X_t(\omega)=\omega(t)$, for all $(t,\omega)\in I\times\mathbb{D}(I;\mathbb{R}^d)$. We write $X^i$ for each coordinate processes of the canonical process $X$, for $i\leq d$.

We endow the Skorokhod space $\mathbb{D}(I;\mathbb{R}^d)$, $d\in\mathbb{N}$, with the right-continuous version $\F_t:=\bigwedge_{\varepsilon>0}\F^o_{t+\varepsilon}$ of the raw, i.e., unaugmented, canonical filtration $\F^o_t:=\sigma(X_s:s\leq t)$ generated by the canonical process $X$ of $\mathbb{D}(I;\mathbb{R}^d)$.

\begin{remark}
The right-continuous version of the raw canonical filtration is needed in the proof of Proposition~\ref{semitight}. Alternatively, we could use the universal completion of the raw canonical filtration; see Proposition~\ref{cor:meas}~(b).
\end{remark}

A stochastic process is understood as a probability measure on the filtered canonical space $\left(\mathbb{D}(I;\mathbb{R}^d),\F_T,(\F_t)_{t\in I}\right)$, where $\F_T:=\bigvee_{t\in I}\F^o_t=\bigvee_{t\in I}\F_t$ and $T=\sup_{t\in I}t\in(0,\infty]$. The family of all probability measures, i.e. c{\`a}dl{\`a}g processes, on $(\mathbb{D}(I;\mathbb{R}^d),\F_T)$ is denoted by $\mathbb{P}(\mathbb{D}(I;\mathbb{R}^d),\F_T)$ and two elements of $\mathbb{P}(\mathbb{D}(I;\mathbb{R}^d),\F_T)$ are identified as usual, i.e., $P=Q$, if (and only if) one has $P[F]=Q[F]$, for all $F\in\F_T$; cf. Section~\ref{sec:skoro}.

\subsection{Semimartingales on the Skorokhod space}

\label{sec:semidef}

We recall some basic concepts of semimartingale theory in the present setting. All semimartingales are assumed c{\`a}dl{\`a}g. We adapt the terminology of Dolinsky and Soner \cite{mete2}, and Guo, Tan and Touzi \cite{touzi} and call a probability measure $Q$ on $(\mathbb{D}(I;\mathbb{R}^d),\F_T)$ a \emph{martingale measure}, if the canonical process $X$ is a martingale on $\left(\mathbb{D}(I;\mathbb{R}^d),\F_T,(\F_t)_{t\in I},Q\right)$, i.e., if $X_s=E_Q[X_t\mid\F_s]$, for all $s<t$ in $I$, including $t=\infty$ in the case that $I=[0,\infty]$. Remark that $X$ is a martingale on $[0,\infty]$ if and only if $X$ is a uniformly integrable martingale on $[0,\infty)$; see e.g. \cite[Theorem~1.1.2.~(2)]{pham}. We say that the canonical process $X$ is \emph{$L^p$-bounded}, for some $p\geq 1$, on $(\mathbb{D}(I;\mathbb{R}^d),\F_T,Q)$, if $\sup_{t\in I}\|X_t\|_{L^p(Q)}<\infty$.

(\emph{Special}) \emph{semimartingale} and \emph{supermartingale measures} are defined similarly to martingale measures. On $\left(\mathbb{D}(I;\mathbb{R}^d),\F_T,(\F_t)_{t\in I},Q\right)$, for a fixed probability measure $Q$, let $\E(Q)$ denote the family of elementary predictable integrands, i.e., the family of adapted c{\`a}gl{\`a}d processes of the form
\begin{equation}
\label{eq:strats}
H^i=H^i_0\one_{\{0\}}+\sum_{k=1}^n H^i_{t^i_{k-1}}\one_{(t^i_{k-1},t^i_k]},\ i\leq d,
\end{equation}
where $n\in\mathbb{N}$, $0= t^i_0\leq t^i_1\leq\cdots\leq t^i_n$ in $I$ and each $H^i_{t^i_{k}}$ is $\F_{t^i_k}$-measurable random variable in $L^\infty(Q)$ satisfying $|H^i_{t^i_{k}}|\leq 1$. For a family $\Q\subset\P\left(\mathbb{D}(I;\mathbb{R}^d),\F_T\right)$, consider the following condition
\[
\lim_{c\rightarrow\infty}\sup_{Q\in\Q}\sup_{H\in\E(Q)}Q\left[|(H\bullet X)_t|>c\right]=0,\ \forall t\in I, \label{ut} \tag{UT}
\]
where
\[
(H\bullet X)_t=\sum_{i=1}^d (H^i\bullet X^i)_t,\ t\in I.
\]
The condition \eqref{ut} was introduced by Stricker in \cite{stricker}. By the classical result of Bichteler, Dellacherie and Mokobodzki, a probability measure $Q$ on $\left(\mathbb{D}(I;\mathbb{R}^d),\F_T\right)$ is an $(\F_t)_{t\in I}$-semimartingale measure if and only if $\Q=\{Q\}$ satisfies the condition \eqref{ut}. The family of process \eqref{eq:strats} generates the predictable $\sigma$-algebra and the condition \eqref{ut} is sometimes called the \emph{predictable uniform tightness condition} (P-UT); see e.g. \cite[Thm.~3.21]{hewangyan}. Remark that, for semimartingale measures, no integrability condition is imposed on $X_0$, i.e., the localization of the local martingale in a canonical semimartingale decomposition is understood in the sense of \cite[Def.~7.1]{hewangyan}; cf. \cite[Rem.~6.3]{jacodshiryaev}. We say that a semimartingale measure $Q$ is \emph{of class $\H^p$}, if, on $\left(\mathbb{D}(I;\mathbb{R}^d),(\F_t)_{t\in I},\F_T,Q\right)$, the canonical process $X$ decomposes to a (local) martingale $M$ and a (predictable) finite variation process $A$, $A_0=0$, such that
\begin{equation}
\label{eq:hfinite}
X=M+A\text{ and }\|X\|_{\H^p(Q)}=\|\|M\|_\infty+\|A\|_\mathbb{V}\|_{L^p(Q)}<\infty.
\end{equation}
Every semimartingale of class $\H^p$, for some $p\geq 1$, is a special semimartingale. To obtain compact statements for quasi- and supermartingales, we introduce two conditions. The first condition is
\[
\label{ub} \tag{UB}
\sup_{Q\in\Q}\sup_{t\in I}\left(E_Q[|X_t|]+\sup_{H\in\E(Q)}E_Q[(H\bullet X)_t]\right)<\infty.
\]
The second condition is the same condition, but the $L^1$-boundedness is strengthened to the uniform integrability of the negative parts, for every $t\in I$, i.e.,
\[
\Q\text{ satisfies \eqref{ub} and }\lim_{c\rightarrow\infty}\sup_{Q\in\Q}E_Q[X^{-}_t\one_{\{X^{-}_t>c\}}]=0,\ \forall t\in I. \label{ui} \tag{UI}
\]
The uniform integrability in \eqref{ui} yields the convergence of the first moments that preserves the supermartingale property; see Proposition~\ref{pro:superlimit}. If we insist that $t^i_n=t$ in \eqref{eq:strats}, then the second supremum in \eqref{ub} is attained, by choosing
\[
H^i_k=\text{sign}(E_Q[X^i_{t^i_k}-X^i_{t^i_{k-1}}\mid\F_{t^i_{k-1}}]),\ 1\leq k\leq n,\ i\leq d,
\]
for which the value of the integral is equal to the $(\F_t)_{t\in I}$-conditional variation of $X^i$ on $[0,t]$,
\begin{equation}
\label{condvar}
\text{Var}^Q_t(X^i):=\sup E_Q\left[|X^i(0)|+\sum_{k=1}^n|E_Q[X^i_{t^i_k}-X^i_{t^i_{k-1}}\mid\F_{t^i_{k-1}}]|\right],\ i\leq d,
\end{equation}
where the supremum is taken over all partitions $0\leq t^i_0\leq t^i_1\leq \cdots\leq t^i_n=t$, $n\in\mathbb{N}$; see e.g. \cite[B,~Appendix~II]{dellacheriemeyer}. A probability measure $Q$ on $\left(\mathbb{D}(I;\mathbb{R}^d),\F_T\right)$ is a \emph{quasimartingale measure} if and only if $\Q=\{Q\}$ satisfies the condition \eqref{ub}; see e.g. \cite[Def.~8.12]{hewangyan}. Moreover, a quasimartingale is an $\H^1$-semimartingale if and only if it is bounded in the $L^{1,\infty}$-norm; cf. \cite[B.VII.~(98.9)]{dellacheriemeyer}. Finally, let us note that we have the following hierarchy.
\begin{equation}
\label{eq:hierarchy}
\textnormal{(UI)}\implies\textnormal{(UB)}\implies\textnormal{(UT)}.
\end{equation}
The first implication is obvious. The second implication follows from Lemma~\ref{lem:lowther}.

\begin{lemma}
\label{lem:lowther}
There exists a constant $b>0$ such that, for any $Q\in\P\left(\mathbb{D}(I),\F_T\right)$, $H\in\E(Q)$ and $c>0$, we have
\begin{equation}
\label{eq:low}
Q[|(H\bullet X)_t|> c]\leq \frac{b}{c}\left(E_Q[|X_t|]+\sup_{H'\in\E(Q)}E_Q[(H'\bullet X)_t]\right),\ t\in I,
\end{equation}
where the right-hand side is possibly infinite.
\end{lemma}

The inequality \eqref{eq:low} is well-known, but we provide the proof for the convenience of the reader in Appendix~\ref{sec:proof}.

A family $\Q\subset\P\left(\mathbb{D}(I;\mathbb{R}^d),\F_T\right)$ is called \emph{$J^1$-tight}, if it is exhausted by a sequence of $J^1$-compact sets; see Appendix~\ref{sec:jone}. Following the classical terminology \cite[Definition~15.48]{hewangyan}, we say that a family $\Q\subset\P\left(\mathbb{D}(I;\mathbb{R}^d),\F_T\right)$ is \emph{$C$-tight}, if it is $J^1$-tight and satisfies
\[
\sup_{Q\in\Q}Q\left[\sup_{s\leq t}|\Delta X_s|> c\right]=0,\ \quad \forall t\in I,\ \forall c>0.
\]
The paths of the canonical process $X$ of $\mathbb{D}(I;\mathbb{R}^d)$ lie on $\mathbb{C}(I;\mathbb{R}^d)$ $Q$-almost surely if and only if $\Q=\{Q\}$ is $C$-tight on $\mathbb{D}(I;\mathbb{R}^d)$. An analogous assertion is true for the H{\" o}lder-continuity.

The Markov property is not preserved by the convergence of finite dimensional distributions, but a stronger property is needed; cf. \cite[Subsection~2.3]{lowther2009}. Following Lowther \cite{lowther2009}, we say that a probability measure $Q\in\P\left(\mathbb{D}(I;\mathbb{R}^d),\F_T\right)$ is \emph{Lipschitz-Markov}, if, for every $s,t\in I$ such that $s<t$, for every bounded Lipschitz continuous function $g:\mathbb{R}^d\rightarrow\mathbb{R}$ with a Lipschitz constant $L(g)\leq 1$, there exists a bounded Lipschitz continuous function $f:\mathbb{R}^d\rightarrow\mathbb{R}$, with a Lipschitz constant $L(f)\leq 1$, such that
\[
f(X_s)=E_{Q}[g(X_t)\mid\F_s]\quad Q\text{-a.s..}
\]
The Lipschitz-Markov property is indeed stronger than the Markov property; consider the sequence of functions $g_n(x)=-n\vee |x|\wedge n$, $n\in\mathbb{N}$, on $\mathbb{R}^d$. The Lipschitz-Markov property was introduced by Kellerer in \cite{kellerer}.

\subsection{The Riesz representation on the canonical space}
\label{sec:riesz}

The Riesz representation theorem for the laws of $\mathbb{D}$-valued random variables, i.e. stochastic processes, requires topological assumptions on the canonical space.

\begin{assumption}
\label{ass:riesz}
The Skorokhod space $\mathbb{D}$ is endowed with a topology, under which $\mathbb{D}$ is a completely regular Radon space and the Borel $\sigma$-algebra coincides with the canonical $\sigma$-algebra.
\end{assumption}

The \emph{strict topology} $\beta_0$ on $\mathbb{C}_b(\mathbb{D})$ is a locally convex topology generated by the family of seminorms
\[
p_g(f):=\|fg\|_\infty,\ f\in\mathbb{C}_b(\mathbb{D}),\ g\in\mathbb{B}_0(\mathbb{D}),
\]
where
\[
\mathbb{B}_0(\mathbb{D}):=\{f\in\mathbb{B}_b(\mathbb{D}):\forall\varepsilon>0\ \exists K^\varepsilon\in\K(\mathbb{D})\text{ s.t. }|f(x)|<\varepsilon\ \forall x\notin K^\varepsilon\}.
\]
The collection of finite intersections of the sets
\begin{equation}
\label{eq:beta0basis}
V_{g,\varepsilon}:=\{f\in\mathbb{C}_b(\mathbb{D}):p_g(f)<\varepsilon\},\ g\in\mathbb{B}_0(\mathbb{D}),\ \varepsilon>0,
\end{equation}
forms a local basis at the origin for the topology $\beta_0$.

The linear space of all $\beta_0$-continuous linear functionals on $\mathbb{C}_b(\mathbb{D})$ is isomorphic to the linear space $\mathbb{M}(\mathbb{D})$ of all countable additive measures of finite total variation on $\mathbb{D}$. More precisely, under Assumption~\ref{ass:riesz}, we have the following Riesz representation theorem on the Skorokhod space.

\begin{lemma}
\label{lem:riesz}
Assume that the Skorokhod space $\mathbb{D}$ satisfies Assumption~\ref{ass:riesz}. Then, any $\mu\in\mathbb{M}(\mathbb{D})$ induces a $\beta_0$-continuous linear functional on $\mathbb{C}_b(\mathbb{D})$ by
\begin{equation}
\label{riesz}
u_\mu(f):=\int f d\mu,\  f\in\mathbb{C}_b(\mathbb{D}),
\end{equation}
and any $\beta_0$-continuous linear functional on $\mathbb{C}_b(\mathbb{D})$ is of the form \eqref{riesz} for some unique $\mu\in\mathbb{M}(\mathbb{D})$, and the one-to-one correspondence $\mu\leftrightarrow u_\mu$ defined by \eqref{riesz} is linear.
\end{lemma}
For the elements of $\mathbb{P}(\mathbb{D})\subset\mathbb{M}^+(\mathbb{D})$, we write
\[
E_Q[f]:=\int fdQ,\ f\in\mathbb{C}_b(\mathbb{D}), \ Q\in\mathbb{P}(\mathbb{D}).
\]
The \emph{weak${}^*$ topology} on $\mathbb{M}(\mathbb{D})$ is a locally convex topology generated by the family of seminorms
\begin{equation}
\label{eq:seminorm}
p_f(\mu):=\left|\int f d\mu \right|,\ f\in\mathbb{C}_b(\mathbb{D}),\ \mu\in\mathbb{M}(\mathbb{D}).
\end{equation}
We write $\mu_\alpha\rightarrow_{w^*}\mu$ for a net $(\mu_\alpha)_{\alpha\in\A}$ with a directed set $\A$ in $\mathbb{M}(\mathbb{D})$ converging in the weak${}^*$ topology to $\mu\in\mathbb{M}(\mathbb{D})$, i.e., if
\begin{equation}
\label{eq:net}
\int f d\mu_\alpha\rightarrow\int f d\mu,\ \forall f\in\mathbb{C}_b(\mathbb{D}).
\end{equation}
In the classical case of a metric space $\mathbb{D}$, the weak${}^*$ topology on the non-negative orthant $\mathbb{M}^+(\mathbb{D})$ is metrizable, i.e., the family of seminorms \eqref{eq:seminorm} can be replaced with a single metric and consequently it is sufficient to consider sequences $(\mu_n)_{n\in\mathbb{N}}$ in \eqref{eq:net} that define a topology; see e.g. \cite[Section~8.3]{bogachev}. Thus, the weak${}^*$ topology coincides with the classical notion of the topology of weak convergence widely used in probability theory, that is, the convergence
\[
E_{Q_n}[ f ]\rightarrow E_Q[f],\ \forall f\in\mathbb{C}_b(\mathbb{D}),
\]
for probability measures $(Q_n)_{n\in\mathbb{N}}$ and $Q$ on a metric space $\mathbb{D}$. Following Sentilles \cite{sentilles}, we write "weak${}^*$" in the place of "weak" to distinguish the topology from the weak topology on the bounded continuous functions in the pairing of Lemma~\ref{lem:riesz}.

\subsubsection{Background}\label{sec:hist}

The classical Riesz representation theorem is stated as a Banach space result for bounded continuous functions vanishing at infinity on a locally compact space. The strict topology $\beta_0$, introduced by Buck in \cite{buck} to locally compact spaces, gives up the Banach space structure, but allows to relax the assumption that the bounded continuous functions are vanishing at infinity. Further observations in the 70's by Giles \cite{giles} and Hoffman-Jorgensen \cite{hoffmanjorgensen} lead to a generalization of the Riesz representation theorem for completely regular spaces; locally compact spaces are completely regular. The weak${}^*$ topology was thoroughly studied by Sentilles in \cite{sentilles} in the case of a general completely regular space. A streamlined proof for Lemma~\ref{lem:riesz} can be found, e.g., in the book of Jarchow \cite{jarchow}. The proof relies on the fact that on a completely regular space every continuous function admits a unique continuous extension to the Stone-{\v C}ech compactification of the space. The fact that the underlying topological space is completely regular ($T_{3{^1/_2}}$) is also \emph{necessary} for the Riesz representation theorem in the sense that the separation axiom cannot be relaxed to a weaker one as there exists examples of regular ($T_3$) spaces on which every continuous function is a constant and on such space the Riesz representation theorem cannot be true; see \cite{herrlich}. However, in our setting it suffices to assume that the space is regular; see Section~\ref{sec:weakstar}.

\section{Main results and examples}

\label{sec:mainres}

A stochastic process is regarded as a probability measure on the canonical space and the family of all probability measures (processes) on the canonical space is endowed with the weak${}^*$ topology~\eqref{eq:seminorm} of the Riesz representation theorem~\eqref{riesz}. We impose the following assumption on the canonical space.

\begin{assumption}
\label{ass:main}
The Skorokhod space is endowed with a regular topology that is weaker than Jakubowski's $S$-topology but stronger than the Meyer-Zheng topology ($MZ$).
\end{assumption}
The $S^*$-topology, introduced in Section~\ref{snot}, meets the previous requirements and is arguably the strongest topology on the Skorokhod space for which the results are true. Indeed, see Theorem~\ref{thm:hierarchy} and Remark~\ref{rem:strongest}.

\subsection{Motivated by the analysis of a problem in finance}

This work was initiated by investigations in robust pricing of derivative contracts of Guo, Tan and Touzi \cite[Lemma~3.7]{touzi} and the current author together with Cheridito, Pr{\" o}mel and Soner \cite[Corollary~6.7]{mart}. In a parallel work \cite{mart}, we describe a general superhedging-sublinear-pricing paradigma and its relation to the compactness studied in the present work.

In our application, the canonical process $X$ presents the value of the underlying asset that can consist of liquid options or common stocks. The objective is to determine the price for a derivative contact $\xi$ that is a function of $X$ i.e. $\xi:\mathbb{D}\rightarrow\mathbb{R}$. Every viable pricing model $Q\in\mathbb{P}(\mathbb{D})$ for a class of $\F_T$-measurable derivative contracts $\mathcal{D}(\F_T)$ on the underlying asset $X$ must satisfy
\[
E_Q[\xi]\leq \Phi(\xi),\ \forall \xi\in \mathcal{D}(\F_T),
\]
where $\Phi(\xi)$ denotes the greatest lower bound for the initial capital requirement at time $t=0$ of all portfolios that produce a value greater or equal to $\xi$ at time $t=T$, for every possible realization of the underlying $X$. Indeed, otherwise it is possible to make a sure profit by creating a portfolio that consists of a short position on a derivative contract $\xi\in\mathcal{D}(\F_T)$ and a long position on a portfolio that produce a value greater or equal to $\xi$. On an efficient market, this should not be possible; cf. the seminal work by Black and Scholes \cite[Abstract]{blackscholes}. On the other hand, for $\xi\in \mathcal{D}(\F_T)$, if the least upper bound $\mathcal{V}(\xi)$ for the expected value of $\xi$ over all viable pricing models is taken to be the lower bound for the price of $\xi$, then one may ask does this lower bound coincide with the upper bound given by the superhedging price given as $\Phi(\xi)$, i.e., one seeks sufficient conditions for the equality
\[
\mathcal{V}(\xi)=\Phi(\xi),\ \forall \xi\in \mathcal{D}(\F_T).
\]
In the case of an increasing, sub-linear $\Phi$, it turns out that the necessary and sufficient condition for the two values to be equal for $\mathcal{D}(\F_T)=\mathbb{C}_b(\mathbb{D})$ is the lower $\beta_0$-semicontinuity of $\Phi$ on $\mathbb{C}_b(\mathbb{D})$. The question whether the set of viable market models is weak${}^*$ compact then arises naturally. Indeed, the weak${}^*$ compactness allows to extended the duality correspondence immediately for $\mathcal{D}(\F_T)=\mathbb{U}_b(\mathbb{D})$ and is also a sufficient condition for the duality to hold on $\mathcal{D}(\F_T)=\mathbb{B}_b(\mathbb{D})$, under another continuity assumption on $\Phi$, namely, the upper $\sigma$-order semicontinuity of $\Phi$ on $\mathbb{B}_b(\mathbb{D})$, provided that the underlying space $\mathbb{D}$ is perfectly normal ($T_6$). These extension criteria are classical in measure transport; consult e.g. the seminal work by Strassen \cite{strassen}.

Though, no probabilistic assumption on the stock dynamics is made a priori, it is reasonable to restrict to the class of semimartingale measures that form the largest class of stock price models for which it is impossible to make unbounded profits (or losses) by selling and buying the stock in a non-anticipative manner. Indeed, this is the statement of the Bicheteler-Dellacherie-Mokobodzki theorem. Further, if one allows non-anticipative trading of the underlying asset without transaction costs and no interest rate, then all viable pricing models are martingale measures on the canonical space of c{\` a}dl{\` a}g paths. If there is static positions available on the market, they translate to additional half-space constraints on the viable martingale measures. This is the so-called martingale optimal transport problem studied e.g. in the aforementioned works of \cite{beiglbock}, \cite{mete2}, \cite{touzi}, \cite{mart}.

\subsection{Main results}\label{sec:maine}

The following Theorem~\ref{thm:semi} is our main result that together with its corollaries and an auxiliary lemma provides an easy method of constructing weak${}^*$ sets of semimartingale measures. The statement regarding sequential compactness in Theorem~\ref{thm:semi} refines the classical results of Meyer and Zheng \cite{meyerzheng}, Stricker \cite{stricker} and Jakubowski \cite{jakubowski} for semimartingale measures, i.e., for semimartingales on the canonical space. The statement about (non-sequential) compactness is, to the best of our knowledge, a new result.

The proofs are postponed to Section~\ref{sec:theproofs}.

\begin{theorem}
\label{thm:semi}
Let $\S$ be a family of semimartingale measures satisfying the condition \eqref{ut}. Under Assumption~\ref{ass:main}, the set $[\S]_{seq}$ is a weak${}^*$ compact and sequentially weak${}^*$ compact set of semimartingale measures.
\end{theorem}
%

\begin{corollary}
\label{cor:quasi}
Let $\Q$ be a family of quasimartingale measures satisfying the condition \eqref{ub}. Under Assumption~\ref{ass:main}, the set $[\Q]_{seq}$ is a weak${}^*$ compact and sequentially weak${}^*$ compact set of quasimartingale measures.
\end{corollary}

\begin{corollary}
\label{cor:super}
Let $\mathcal{M}$ be a set of supermartingale measures satisfying the condition \eqref{ui}. Under Assumption~\ref{ass:main}, the set $[\mathcal{M}]_{seq}$ is a weak${}^*$ compact and sequentially weak${}^*$ compact set of supermartingale measures.
\end{corollary}


By combining one or both of the assertions of the following Lemma~\ref{lem:cm} with Theorem~\ref{thm:semi}, or one of its corollaries, one obtains compact sets of continuous and Markov semimartingales, quasimartingales and supermartingales.

\begin{lemma}\label{lem:cm}
Let $\mathcal{P}$ be a family of probability measures on the Skorokhod space. Under Assumption~\ref{ass:main}, we have the following:
\begin{enumerate}[(a)]
\item If the set $\mathcal{P}$ is $C$-tight, then the set $[\mathcal{P}]_{seq}$ consists of continuous processes.
\item If each measure in the set $\mathcal{P}$ is Lipschitz-Markov, then the set $[\mathcal{P}]_{seq}$ consists of Lipschitz-Markov processes.
\end{enumerate}
\end{lemma}

\subsection{Examples}

The following Example~\ref{ex:mart}, essentially an observation made by Guo, Tan and Touzi \cite[Lemma~3.7]{touzi}, was our original motivation to study weak${}^*$ compactness in the present setting. In Example~\ref{ex:mart}, we allow an infinite time horizon i.e. the index set $I=[0,T]$ for some $T\in(0,\infty]$.

\begin{example}
\label{ex:mart}
 Let $\mathcal{M}^u$ be the family of u niformly integrable ($L^1$-bounded) martingale measures and let $\mathcal{P}$ be a weak${}^*$ compact subset of $\mathbb{P}(\mathbb{R}^d)$, that is, the family of probability measures on $\mathbb{R}^d$. Then, the set
\[
\mathcal{M}^u_\mathcal{P}=\{Q\in\mathcal{M}^u:Q\circ X^{-1}_T\in\mathcal{P}\}
\]
is weak${}^*$ compact and sequentially weak${}^*$ compact.
\end{example}
\begin{proof}
We adapt the proof of \cite[Lemma~3.7]{touzi}. For $a>0$, we have
\[
E_{Q}[|X_t|\one_{\{|X_t|\geq a\}}] \leq  2E_{Q}[(|X_t|-a/2)^+]\leq 2 E_Q[(|X_T|-a/2)^+]
\]
uniformly over $(t,Q)\in I\times\mathcal{M}^u_\mathcal{P}$, and
\[
E_Q[(H\bullet X)_t]=0,
\]
for every (elementary) predictable $|H|\leq 1$, for every $t\in I$, for every $Q\in\mathcal{M}^u$. Thus, by the general form of Prokhorov's theorem, see e.g. \cite[Theorem~8.6.2.]{bogachev}, the family $\mathcal{M}^u_\mathcal{P}$ satisfies the condition \eqref{ui}; cf. \cite[IX,~Lemma~1.11]{jacodshiryaev}. By Example~\ref{contexamples}~(b), the evaluation mapping is (sequentially) continuous at the terminal time, so, we have $\mathcal{M}^u_\mathcal{P}=[\mathcal{M}^u_\mathcal{P}]_{seq}$. A measure $Q$ is a martingale measure for $X$ on $\mathbb{D}(I;\mathbb{R}^d)$ if and only if $Q$ is a supermartingale measure for $X^i$ and $-X^i$, for every $i\leq d$, so, by Corollary~\ref{cor:super}, the set $\mathcal{M}^u_\mathcal{P}$ is weak${}^*$ compact and sequentially weak${}^*$ compact.
\end{proof}

\begin{example}
\label{boundmart}
Let $\mathcal{M}^p$ denote the family of $L^p$-bounded martingale measures. Then, the sets
\begin{equation}
\label{eq:sets}
\mathcal{M}^p_r:=\{Q\in\mathcal{M}^p:\|X\|_{L^{p,\infty}(Q)}\leq r\},\ r\in\mathbb{R}_+,
\end{equation}
are weak${}^*$ compact and sequentially weak${}^*$ compact, for $1< p<\infty$.
\end{example}
\begin{proof}
An increasing continuous function $y\mapsto y^p$ composed with a lower semicontinuous function $y=\|\omega\|_\infty$ is lower semicontinuous, see Lemma~\ref{lem:lsc}, and non-negative, so, by \cite[Pro.~8.9.8.]{bogachev}, the functional $\|X\|_{L^{p,\infty}(Q)}$ is lower semicontinuous in the weak${}^*$ topology. Thus, the set $\mathcal{M}_r^p$ is weak${}^*$ closed, for $r>0$ and $p>1$. The $L^p$-boundedness, for $p>1$, implies that the set $\mathcal{M}_r^p$ satisfies the condition~\eqref{ui}, for $r>0$ and $p>1$, so, the set $\mathcal{M}_r^p$ is weak${}^*$ compact and sequentially weak${}^*$ compact, for $r>0$ and $p>1$; cf. Example~\ref{ex:mart}.
\end{proof}

Assume that a probability measure $Q$ is fixed and $p>1$. Then the Hardy space of $L^p(Q)$-bounded (equivalence classes of indistinguishable) c{\`a}dl{\`a}g martingales $\mathcal{M}^p(Q):=\mathcal{M}^p(\mathbb{D}(I;\mathbb{R}),\F_T,(\F_t)_{t\in I},Q)$ can be identified with the Lebesgue space $L^p(Q):=L^p(\mathbb{D}(I;\mathbb{R}),\F_T,Q)$. Indeed, there exists a linear one-to-one correspondence between the (uniformly integrable) $L^p(Q)$-bounded martingales on $[0,T]$ and the random variables $X_T$ of $L^p(Q)$, $p>1$, as each $X_T\in L^p(Q)$ defines a $Q$-a.s. unique c{\`a}dl{\`a}g $L^p(Q)$-bounded martingale on $[0,T]$ via
\[
X_t:=E_Q[X_T\mid\F_t],\ t\in[0,T],
\]
and vice versa, each such c{\`a}dl{\`a}g martingale has $Q$-a.s. unique terminal value $X_T\in L^p(Q)$. By Doob's $L^p$-maximal inequality, the $L^{p,\infty}(Q)$-norm on $\mathcal{M}^p(Q)$ and the $L^p(Q)$-norm on $L^p(Q)$ are equivalent, for $p>1$. In the case $p=1$, there is no one-to-one correspondence, but we only have $\mathcal{M}^1(Q)\subset L^1(Q)$, for $I=[0,\infty)$. Further discussion on this one-to-one correspondence can be found e.g. in \cite[B.VII.64]{dellacheriemeyer}. Let us mention \cite[Theorem~3]{stricker2}, where the one-to-one correspondence is used in the construction of semimartingale decomposition for quasimartingales, that is a central step in the classical proofs for the Bichteler-Dellacherie-Mokobodzki theorem.

Due to the one-to-one correspondence, the compactness follows from the classical results for Banach spaces. For $p>1$, the space $L^p(Q)$ is a reflexive Banach space, so, the (sequential) weak${}^*$ compactness of the sets \eqref{eq:sets} follows from the Banach-Alaoglu theorem in conjunction with the Eberlein-{\v S}mulian theorem; see \cite{stricker}. The Dunford-Pettis theorem states that a uniformly integrable subset of a non-reflexive Banach space $L^1(Q)$ is relatively sequentially compact in the weak topology, but the random variables of $L^1(Q)$ are not in one-to-one correspondence with neither the family of $L^{1,\infty}(Q)$-bounded, nor $L^{1}(Q)$-bounded martingales, for $I=[0,\infty)$.

\begin{example}
\label{ex:semi}
Let $\H^p$ denote the family of $\H^p$-semimartingale measures. Then, the sets
\begin{equation}
\label{hpllset}
\S^p_r:=\{Q\in\H^p:\|X\|_{L^{p,\infty}(Q)}+\|X\|_{\E^p(Q)}\leq r\},\ r>0,
\end{equation}
are weak${}^*$ compact and sequentially weak${}^*$ compact, for $1\leq p<\infty$.
\end{example}
\begin{proof}
The sets $\S^p_r$, $r> 0$, $p\geq 1$, satisfy the condition \eqref{ub}, so, by Corollary~\ref{cor:quasi}, the sets $[\S^p_r]_{seq}$ are weak${}^*$ compact and sequentially weak${}^*$ compact sets of quasimartingales. Moreover, for any sequence $(Q_n)_{n\in\mathbb{N}}$ in $\S^p_r$ converging in the weak${}^*$ topology to some $Q$, we have
\begin{equation}
\label{eq:unilsc}
\| X \|_{L^{p,\infty}(Q)}\leq\liminf_{n\rightarrow\infty}\| X \|_{L^{p,\infty}(Q^n)}<\infty,
\end{equation}
for $p\geq 1$; cf. Example~\ref{boundmart}. Thus, we have $[\S^p_r]_{seq}\subset\H^1$; cf. \cite[B.VII.98]{dellacheriemeyer}. To show that $\S^p_r=[\S^p_r]_{seq}$ and $[\S^p_r]_{seq}\subset\H^p$, we introduce an auxiliary class $\A$ of smooth elementary integrands of the form
\begin{equation}
\label{eq:smooth}
A^i=\sum_{j=1}^k A^i_{t^i_{j-1}}\varphi_{t^i_{j-1},t^i_j},\ i\leq d,
\end{equation}
where $k\in\mathbb{N}$, $0= t^i_0\leq t^i_1\leq\cdots\leq t^i_k$ in $I$ and each $A^i_{t^i_{j}}$ is \emph{continuous} $\F_{t^i_j-}$-measurable function satisfying $|A^i_{t^i_{j}}|\leq 1$ and each $\varphi_{t^i_{j-1},t^i_j}$ is a \emph{smooth} function on $I$ vanishing outside $(t^i_{j-1},t^i_j+\varepsilon^i_j)$, for some $\varepsilon^i_j\in(t^i_j,t^i_{j+1})$, and satisfies $|\varphi_{t^i_{j-1},t^i_j}|\leq 1$; we allow non-zero $\varphi_{t^i_{k-1},t^i_k}(t^i_k)\in[-1,1]$ for $t^i_k=T$; cf. \eqref{inmeas}. Now, let $Q\in[\S_r^p]_{seq}$ and assume that we are given an elementary predictable integrand $H=(H^i)_{i=1}^d$, i.e., an element of $\E(Q)$, see \eqref{eq:strats}, such that each $H^i_{t^i_j}$ in \eqref{eq:strats} is $\F^o_{t^i_j-}$-measurable. By \cite[A.IV.69~(c)]{dellacheriemeyer}, the domain of $\F^o_{t^i_j-}$-measurable functions is homeomorphic to a closed subset of $\mathbb{D}(I;\mathbb{R}^d)$; cf. Corollary~\ref{minusfilt}. Thus, by Lusin's theorem in conjunction with Tietze's extension theorem, for every $\F^o_{t^i_j-}$-measurable $|H^i_{t^i_j}|\leq 1$, there exists a sequence of continuous $\F^o_{t^i_j-}$-measurable functions $(A^{i,n}_{t^i_j})_{n\in\mathbb{N}}$ with $|A^{i,n}_{t^i_j}|\leq 1$ and such that $A^{i,n}_{t^i_j}\rightarrow H^i_{t^i_j}$ $Q$-a.s., as $n\rightarrow\infty$; see e.g. \cite{feldman} and \cite[2.1.8.]{engelking}. Moreover, c{\`a}gl{\`a}d step functions can be approximated from right with smooth functions (and vice versa), so, there exists a sequence $(A^n)_{n\in\mathbb{N}}$, $A^n=(A^{i,n})_{i=1}^d$, of elements of $\A$ such that $\|A^n\|_\mathbb{V}\leq\|H\|_\mathbb{V}$ $Q$-a.s., for all $n\in\mathbb{N}$, and $(A^n\bullet X)_T\rightarrow (H\bullet X)_T$ $Q$-a.s., as $n\rightarrow\infty$. Integrating by parts, we get
\begin{equation}
\label{eq:upperbound}
|(A^n\bullet X)_T|\leq \left(|X_TA^n_T|+\|X\|_\infty \|A^n\|_\mathbb{V}\right)\leq c\|X\|_\infty,\ A^n_0=0,\ n\in\mathbb{N},
\end{equation}
where $c:=2\|H\|_\mathbb{V}<\infty$ $Q$-a.s. and $\|X\|_\infty\in L^p(Q)$, by \eqref{eq:unilsc}. Thus, by the dominated convergence, for any $Q\in[S^p_r]_{seq}$, we have
\[
\lim_{n\rightarrow\infty}\|(A^n\bullet X)_T\|_{L^p(Q)}=\|(H\bullet X)_T\|_{L^p(Q)}.
\]
The elements of $\A$ can similarly be approximated with the elements of $\E(Q)$. Moreover, due to the uniform bound \eqref{eq:upperbound}, for any sequence of integrands bounded in total variation, by the right-continuity $X=(X^1,X^2,\dots,X^d)$, the $\F^o_{t^i_j-}$-measurability of the random variables $H_{t^i_j}$ can be relaxed to $\F^o_{t^i_j}$-measurability, and further, to $\F^o_{t^i_j+}$-measurability; cf. \eqref{eq:minusplus}. Thus, for any $Q\in[\S_r^p]_{seq}$, we have
\[
\|X\|_{\E^p(Q)}=\|X\|_{\A^p(Q)}:= \sup_{A\in\A}\|(A\bullet X)_T\|_{L^p(Q)}.
\]
Now, since each $A^i$ of $A=(A^i)_{i=1}^d$ in $\A$ is continuously differentiable in $t$, for every $\omega\in\mathbb{D}(I;\mathbb{R}^d)$, the function $|(A\bullet X)_T|^p$ is continuous, see \eqref{inmeas}, and non-negative, so, by Proposition~\ref{eqvcomp} and \cite[Proposition~8.9.8.]{bogachev}, the functional $\|(A\bullet X)_T\|_{L^p(Q)}$ is weak${}^*$ lower semicontinuous on $\S^p_r$, for every $A\in\A$. Consequently, the functional $\|X\|_{\A^p(Q)}$ is weak${}^*$ lower semicontinuous on $\S^p_r$, which in conjunction with the weak${}^*$ lower semicontinuity \eqref{eq:unilsc} of the functional $\|X\|_{L^{p,\infty}(Q)}$, yields the weak${}^*$ closedness of the sets $\S^p_r$ in $\E^p$. Indeed, for any $r>0$ and $p\geq 1$, for any sequence $(Q_n)_{n\in\mathbb{N}}$ in $\S^p_r$, converging in the weak${}^*$ topology to some $Q$, we have
\[
\begin{split}
\| X \|_{L^{p,\infty}(Q)}+\|X\|_{\E^p(Q)}&=\| X \|_{L^{p,\infty}(Q)}+\|X\|_{\A^p(Q)}\\
&\leq\liminf_{n\rightarrow\infty}\left(\| X \|_{L^{p,\infty}(Q_n)}+\|X\|_{\A^p(Q_n)}\right)\\
&=\liminf_{n\rightarrow\infty}\left(\| X \|_{L^{p,\infty}(Q_n)}+\|X\|_{\E^p(Q_n)}\right)\leq r,
\end{split}
\]
i.e., $Q\in\S_r^p$. Thus, we have $\S^p_r=[\S^p_r]_{seq}\subset \H^p$, i.e., the sets $\S^p_r$ are weak${}^*$ compact, for $r>0$ and $p\geq 1$. Every element in $\S^p_r$ is indeed an $\H^p$-semimartingale measure; cf.~\eqref{hpnorm}-\eqref{eq:hardy}.

\end{proof}

The pseudonorm in Example~\ref{ex:semi}, given by the sum of the ${L^{p,\infty}}$-norm and the Emery pseudonorm
\begin{equation}
\label{hpnorm}
\|X\|_{\E^p(Q)}:=\sup_{H\in\E(Q)}\|(H\bullet X)_T\|_{L^p(Q)},\ p\geq 1,
\end{equation}
is equivalent to the (maximal) $\E^p$-norm
\begin{equation}
\label{eq:hardy}
\|X\|_{\E^p(Q)}:=\|\|M\|_\infty+\|A\|_\mathbb{V}\|_{L^p(Q)},\ p\geq 1,
\end{equation}
where $X=M+A$, $A_0=0$, denotes the canonical semimartingale decomposition of $X$ under $Q$; see \cite[B.VII.104]{dellacheriemeyer} and \cite[B,~p.305]{dellacheriemeyer}. Assume that a probability measure $Q$ is fixed and $p>1$. Then, the Hardy space of $\H^p(Q)$-bounded (equivalence classes of indistinguishable) semimartingales $\H^p(Q):=\H^p(\mathbb{D}(I;\mathbb{R}),\F_T,(\F_t)_{t\in I},Q)$ is a Banach space; see \cite[p.292]{hewangyan}. For martingales in particular, the norm $\|X\|_{\E^p(Q)}$ is equivalent to the norm $\|X\|_{L^{p,\infty}(Q)}$, and, as mentioned in the context of Example~\ref{boundmart}, there is an analogous Banach pairing $(\mathcal{M}^p(Q))'=(L^p(Q))'=L^q(Q)=\mathcal{M}^q(Q)$, for $p,q>1$, $1/p+1/q=1$; see \cite[B,~p.253]{dellacheriemeyer} and \cite[p.281]{hewangyan}.

In contrast to the classical Banach space pairing, in the pairing of the Riesz representation theorem it is straightforward to construct compact sets of continuous and Markov processes by invoking the classical stability criteria for the almost sure (H{\" o}lder) continuity and the (Lipschitz) Markov property.

\begin{example}\label{ex:new}
Let $\mathcal{S}$ be a set of semimartingale measures satisfying the condition~\eqref{ut}.
\begin{enumerate}[(a)]
\item Assume that there exist  constants $a,b,c>0$ such that
\begin{equation}
\label{eq:kolmogorov}
\sup_{Q\in\S} E_{Q}[|X_t-X_s|^a]\leq b|s-t|^{1+c},\quad 
\forall s,t\in I.
\end{equation}
Then, the set $[\S]_{seq}$ is a weak${}^*$ compact set of continuous semimartingales.
\item Let $Q$ be the standard Wiener measure on the Skorokhod space $\mathbb{D}(\mathbb{R}_+;\mathbb{R})$. Assume that, for every $Q^\alpha$ in $\mathcal{P}$, there exists a function $\sigma^\alpha:\mathbb{R}_+\times\mathbb{R}\rightarrow\mathbb{R}$ that is continuous, continuously differentiable in the first component and that $Q^\alpha$ is the law of a (unique strong) solution $X^\alpha$ of
\begin{equation}
\label{eq:sde}
X^\alpha_0\in\mathbb{R};\quad X^\alpha_t=X^\alpha_0+\int_0^t \sigma^\alpha(u,X^\alpha_u) d X_u,\quad \forall t\geq 0,\ Q\text{-a.s..}
\end{equation}
Then, the set $[\S]_{seq}$ is a weak${}^*$ compact set of Markov semimartingales.
\end{enumerate}
Analogous assertions are true for the sequential closures $[\mathcal{Q}]_{seq}$ and $[\mathcal{M}]_{seq}$ of a set $\mathcal{Q}$ of quasimartingales measures and a set $\mathcal{M}$ of supermartingale measures satisfying the condition \eqref{ub} and \eqref{ui}, respectively.
\end{example}
\begin{proof}
(a) The Kolmogorov continuity criteria \eqref{eq:kolmogorov} is a sufficient criteria for $C$-tightness of $\S$; see e.g. \cite[XII.1.8]{revuzyor}. Thus, the statement follows from Theorem~\ref{thm:semi} in conjunction with Lemma~\ref{lem:cm}~(a).

(b) We adapt the argument of \cite[Proposition~5]{hirsch}. Under the assumption, an equation of the form~\eqref{eq:sde} admits a unique strong solution and for a fixed $\alpha$, $s\geq 0$ and $x\in\mathbb{R}$, let $X^{\alpha,s,x}=(X^{\alpha,s,x}_t)_{t\geq s}$ denote the solution to
\[
X^{\alpha,s,x}_s=x;\quad X^{\alpha,s,x}_t=x+\int_s^t \sigma^\alpha(u,X^{\alpha,u,x}_u) d X_u,\quad \forall t\geq s,\ Q\text{-a.s.}
\]
and let $M^{\alpha,s,x}=(M^{\alpha,s,x}_t)_{t\geq s}$ denote the process $M^{\alpha,s,x}_t:=\frac{\partial}{\partial x} X^{\alpha,s,x}_t$, that is, for every $t\geq s$, equal to
\[
M^{\alpha,s,x}_t=\exp\left( \int_s^t  \sigma^\alpha_x(u,X^{\alpha,s,x}_u)dX_u-\frac{1}{2}\int_s^t \left(\sigma^\alpha_x\right)^2(u,X^{\alpha,s,x}_u)d u\right)\ Q\text{-a.s.,}
\]
where $\sigma^\alpha_x:=\frac{\partial}{\partial x}\sigma^\alpha$. The process 
$M^{\alpha,s,x}$ is a non-negative local martingale, so, $M^{\alpha,s,x}_t\geq 0$ $Q$-a.s. and $E_Q[M^{\alpha,s,x}_t]\leq 1$, for every $t\geq s$. Assume now that $g$ is a bounded continuously differentiable function with $|g'(x)|\leq 1$, for every $x\in\mathbb{R}$, and fix $t> s\geq 0$. Then, define a function $f:\mathbb{R}\rightarrow\mathbb{R}$ as
\[
f(x):=E_Q[g(X^{\alpha,s,x}_t)],\quad x\in\mathbb{R}.
\]
The function $f$ is a bounded continuously differentiable function with
\[
|f'(x)|=E_Q[g'(X^{\alpha,s,x}_t)M^{\alpha,s,x}_t]\leq 1,\quad x\in\mathbb{R}.
\]
Hence, $X^{\alpha}=(X^{\alpha}_t)_{t\geq 0}$ satisfies the Lipschitz-Markov property on the space $(\mathbb{D}(\mathbb{R}_+;\mathbb{R}),\F_\infty,(\F_t)_{t\geq 0},Q)$, for every $\alpha$, i.e., the law $Q^\alpha$ of $X^\alpha$ is a Lipschitz-Markov measure on the space $(\mathbb{D}(\mathbb{R}_+;\mathbb{R}),\F_\infty,(\F_t)_{t\geq 0})$, for every $\alpha$. Thus, by Theorem~\ref{thm:semi} in conjunction with Lemma~\ref{lem:cm}~(b), every element of $[\mathcal P]_{seq}$ is a (Lipschitz) Markov semimartingale measure.

The statements for quasimartingale measures and supermartingale measures are obtained by replacing Theorem~\ref{thm:semi} above with Corollary~\ref{cor:quasi} and Corollary~\ref{cor:super}, respectively.
\end{proof}

\section{Auxiliary results and the proofs for Section~\ref{sec:maine}}
\label{sec:basicresults}

The purpose of this section is to provide the proofs for Theorem~\ref{thm:semi} and its corollaries that we omitted in Section~\ref{sec:maine} and the required auxiliary results leading to the proofs. In Section~\ref{sec:weakstar}, we establish three basic results for the weak${}^*$ topology deployed in the proof of Theorem~\ref{thm:semi}. The required stability and tightness results for the weak${}^*$ topology are covered in Section~\ref{sec:tightstab}. Finally, the proofs for the results of Section~\ref{sec:maine} are provided in Section~\ref{sec:theproofs}.

In Section~\ref{sec:weakstar}, in addition to that the underlying topological space is regular and Souslin, we assume the space posses the following property.

\begin{property}
\label{separatingfamily}
There exists a countable family of real-valued continuous functions $f_k$, $k\in\mathbb{N}$, such that, for all $x,y\in X$, we have
\begin{equation}
\label{eq:sf}
f_k(x)=f_k(y),\ \forall k\in\mathbb{N}\implies x=y.
\end{equation}
\end{property}

\begin{remark}
A topological space satisfying Property~\ref{separatingfamily} is submetrizable i.e. there exists a weaker topology that is metrizable. Indeed, such topology is given e.g. by the metric
\begin{equation}
\label{eq:metric}
d(x,y):=\sum_{k=1}^\infty 2^{-k}\frac{|f_k(x)-f_k(y)|}{1+|f_k(x)-f_k(y)|},\quad x,y\in X,
\end{equation}
where $f_k$'s are given by \eqref{eq:sf}. The author would like to thank Professor Jakubowski for pointing out this fact.
\end{remark}

Property~\ref{separatingfamily} was extensively studied in \cite{jakubowski5}; see also Section~\ref{sec:sprop}. In particular, we notice that a Skorokhod space satisfying Assumption~\ref{ass:main} has this property. Indeed, for a Skorokhod space $\mathbb{D}(I;\mathbb{R}^d)$, a countable family of separating continuous functions is given e.g. by the family of functions
\[ 
\omega\mapsto\frac{1}{r}\int_q^{q+r}\omega^i(t)dt\text{ and }\omega\mapsto\omega^i(T),\text{ for }I=[0,T],
\]
where $q$ and $q+r$ run over the rationals in $I$ and $i$ over the spatial dimensions $1,\dots,d$; cf. \eqref{inmeas} and \eqref{interm}. Therefore, the results of this section are true for a Skorokhod space satisfying Assumption~\ref{ass:main}. Indeed, the Souslin property is verified in Proposition~\ref{pro:hierarchy}.

\subsection{Weak${}^*$ topology}
\label{sec:weakstar}

The results of this section are established under the assumption that \emph{the space $\mathbb{D}$ is a regular Souslin space satisfying Property~\ref{separatingfamily}}. Under the assumption, we obtain a stronger separation axiom than the required $T_{3{^1/_2}}$; cf. Section~\ref{sec:riesz}. Indeed, combining the fact that the topological space is regular ($T_{3}$) with the fact that the space is a Souslin space, it follows, from a result of Fernique \cite[Proposition~I.6.1]{fernique}, that \emph{the space $\mathbb{D}$ is perfectly normal} ($T_6$).

The families $\M_t(\mathbb{D})$, $\M_\tau(\mathbb{D})$ and $\M_\sigma(\mathbb{D})$ are defined at the end of Section~\ref{sec:not}.

\begin{proposition}
The following characterize the dual space in the pairing of Lemma~\ref{lem:riesz}.
\label{cor:meas}
\begin{enumerate}[(a)]
\item We have that
\begin{equation}
\label{eq:meas}
\M_t(\mathbb{D})=\M_\tau(\mathbb{D})=\M_\sigma(\mathbb{D}).
\end{equation}
\item The dual of $\mathbb{C}_b(\mathbb{D})$ can be identified with the class of measures \eqref{eq:meas} on (the universal completion of) the canonical $\sigma$-algebra under the bilinear form~\eqref{riesz}.
\end{enumerate}
\end{proposition}

\begin{proof}
(a) Every Lusin space is a Radon space; see e.g. \cite[p.122]{schwartz}. Thus, we have $\M_\sigma(\mathbb{D})\subset\M_t(\mathbb{D})$. The equality \eqref{eq:meas} follows from the fact that the inclusions $\M_t(\mathbb{D})\subset\M_\tau(\mathbb{D})$ and $\M_\tau(\mathbb{D})\subset\M_\sigma(\mathbb{D})$ are true for an arbitrary topological space; see \cite[Proposition~7.2.2.]{bogachev}.

(b) Let $\widehat{\B}(\mathbb{D})$ denote the universal completion of the Borel $\sigma$-algebra $\B(\mathbb{D})$. For every $\mu\in\P(\B(\mathbb{D}))$, there exists a unique $\widehat{\mu}\in\P(\widehat{\B}(\mathbb{D}))$ such that
\[
\int f d\mu =\int f d\widehat{\mu},\ \forall f\in\mathbb{C}_b(\mathbb{D}).
\]
Since any measure of finite variation is a linear combination of two probability measures, it suffices to observe that the mapping $\mu\mapsto\widehat{\mu}$ is a bijection; see e.g. \cite[A,~32~(c)~(i)]{dellacheriemeyer}. The statement then follows from (a) in conjunction with \cite[Theorem~7.6.3.]{jarchow}.
\end{proof}

\begin{remark}
It also follows that every measure on the class \eqref{eq:meas} is perfect; see e.g \cite[Theorem~7.5.10.~(i)]{bogachev}.
\end{remark}

We use the equality \eqref{eq:meas} without mentioning it when we apply the results from the book of Bogachev \cite{bogachev}.

\subsubsection{The Eberlein-{\v S}mulian properties}
\label{sec:es}


In this section, we show that the non-negative orthant $\M_+(\mathbb{D})$ endowed with the weak${}^*$ topology is \emph{angelic}. In angelic spaces, the properties of compactness and sequential compactness coincide. In general, one does not imply another; see e.g. \cite[Example~3.4.1]{bogachevtvs}.

\begin{proposition}
\label{eqvcomp}
The space of non-negative measures $\M_+(\mathbb{D})$ endowed with the weak${}^*$ topology is angelic. In particular, for any subset $M\subset\M_+(\mathbb{D})$, the following are equivalent:
\begin{enumerate}[(i)]

\item Any infinite sequence in $M$ has a weak${}^*$ convergent subsequence in $\M_+(\mathbb{D})$,

\item The weak${}^*$ closure of $M$ is weak${}^*$ compact in $\M(\mathbb{D})$.

\end{enumerate}
Moreover, under these conditions, the weak${}^*$ closure of $M$ is metrizable.
\end{proposition}
\begin{proof}
By the assumption that the underlying topological space $\mathbb{D}$ is a regular Souslin space, it admits a continuous injective mapping to a metric space; see \cite[Theorem~2.25~(i)]{florescu}. It is also known that if a regular space can be continuously injected into an angelic space, then this regular space is also angelic; see \cite[Theorem~3.3]{floret}. Since the weak${}^*$ topology on the space of non-negative measures $\M_+(\mathbb{D})$ is metrizable for a metrizable topology on the underlying space $\mathbb{D}$, see \cite[Theorem~8.3.2]{bogachev}, and metric spaces are angelic, the space $\M_+(\mathbb{D})$ endowed with the weak${}^*$ topology is angelic under our assumption that the underlying topological space $\mathbb{D}$ is a regular Souslin space. By \cite[Theorem~8.10.4.]{bogachev}, the weak${}^*$ closure of a subset $M$ of $\M_+(\mathbb{D})$ satisfying (i) or (ii) is a compact metrizable subspace of $\mathbb{M}(\mathbb{D})$, so, (i) and (ii) are indeed equivalent for $M$.
\end{proof}

It is immediate from Proposition~\ref{eqvcomp} that the properties of weak${}^*$ compactness and sequential weak${}^*$ compactness are equivalent for the subsets of $\M_+(\mathbb{D})$. In fact, a stronger statement is true. In angelic space, the closure of a relatively compact set is completely exhausted by the limits of sequences of points in this set.

\begin{corollary}
\label{cor:weakseq}
Assume that $M$ is a subset of $\M_+(\mathbb{D})$ that satisfies the equivalent conditions of Proposition~\ref{eqvcomp}. Then the sequential weak${}^*$ closure of $M$ in $\mathbb{M}_+(\mathbb{D})$, i.e., the set
\[
[M]_{seq}=\{\mu\in \mathbb{M}_+(\mathbb{D}):\exists(\mu_n)_{n\in\mathbb{N}}\subset
 M\text{ s.t. }\mu^n\rightarrow_{w^*}\mu\},
\]
is weak${}^*$ closed.
\end{corollary}
\begin{proof}
By Proposition~\ref{eqvcomp}, the closure of $M$ endowed with the relative topology is a first countable space. In particular, the space is a Fr{\' e}chet-Urysohn space. By \cite[Theorem~1.6.14.]{engelking}, the sequential closure $[M]_{seq}$ coincides with the closure of $M$.
\end{proof}

Various criteria that guarantee tightness and stability of a family of processes are not preserved in the weak${}^*$ convergence, so, the previous results are crucial for constructing weak${}^*$ compact sets of stochastic processes.

\subsubsection{Prokhorov's theorem}

We say that a subset $M$ of $\M(\mathbb{D})$ is \emph{$\beta_0$-equicontinuous} if and only if the respective family of linear functionals
\[
\left\{f\mapsto u_\mu(f):=\int fd\mu: \mu\in \M(\mathbb{D})\right\}
\]
is equicontinuous in the $\beta_0$-topology on $\mathbb{C}_b(\mathbb{D})$, i.e., if, for every $\varepsilon>0$, there exists a $\beta_0$-neighbourhood $V$ in $\mathbb{C}_b(\mathbb{D})$ such that $|u_\mu(f)|<\varepsilon$, for all $(f,\mu)\in V\times M$. 

A measure $\mu\in\mathbb{M}(\mathbb{D})$ is called \emph{tight}, if there exists an exhausting net of compact sets $(K^\varepsilon)_{\varepsilon>0}$ for $\mu$, i.e., $|\mu|(\mathbb{D}\setminus K^\varepsilon)<\varepsilon$, for every $\varepsilon>0$, where $|\mu|$ is the total variation of $\mu$. A subset $M$ of $\M(\mathbb{D})$ is called \emph{uniformly tight}, if there exists a net of compact sets $(K^\varepsilon)_{\varepsilon>0}$ which is uniformly exhausting for the total variation of $M$, i.e., $\sup_{\mu\in M}|\mu|(\mathbb{D}\setminus K^\varepsilon)<\varepsilon$, for every $\varepsilon>0$; cf. \eqref{eq:beta0basis}.

\begin{proposition}
\label{prohoskoro}
A subset $M$ of $\M(\mathbb{D})$ is $\beta_0$-equicontinuous if and only if it is bounded in total variation and uniformly tight, and we have
\begin{enumerate}[(a)]
\item If $M$ is a $\beta_0$-equicontinuous subset of $\M(\mathbb{D})$, then $M$ is relatively compact and relatively sequentially compact in the weak${}^*$ topology.
\end{enumerate}
Moreover, we have the following useful convergence criteria.
\begin{enumerate}[(a)]
\setcounter{enumi}{1}
\item If $(\mu_n)_{n\in\mathbb{N}}$ is an uniformly tight sequence in $\M(\mathbb{D})$ converging in the weak${}^*$ topology to $\mu\in\M(\mathbb{D})$, then, for any $f\in\mathbb{C}(\mathbb{D})$ satisfying 
\[
\lim_{c\rightarrow\infty}\sup_{n}\int|f| \one_{\{|f|\geq c\}}d\mu_n=0,
\]
we have
\[
\int fd\mu_n\rightarrow \int fd\mu,
\quad\textnormal{ as }n
\rightarrow\infty.
\]
\end{enumerate}
\end{proposition}
\begin{proof}
The underlying topological space is completely regular, and the characterization follows directly from \cite[Theorem~5.1]{sentilles}. The compact subsets of a completely regular Souslin space are metrizable, cf. Proposition~\ref{eqvcomp} and \cite[Lemma~8.9.2.]{bogachev}, which, in conjunction with the fact the space is completely regular, verifies the assumptions for both, the sequential and the non-sequential, Prokhorov's theorem~(a); see \cite[Theorem~8.6.7.]{bogachev}. The convergence criteria (b) is similarly a direct consequence of the fact that the underlying space is completely regular; see \cite[Lemma~3.8.7.]{bogachevgaussian}.
\end{proof}

The characterization of the relative compactness in terms of the property of $\beta_0$-equicontinuity yields also a criteria for compactness of closures (of convex (circled) hulls).

\begin{corollary}
\label{cor:equi}
The closed convex circled hull of a $\beta_0$-equicontinuous subset of $\mathbb{M}(\mathbb{D})$ is $\beta_0$-equicontinuous, weak${}^*$ compact and sequentially weak${}^*$ compact. In particular, the closure and the closed convex hull of a $\beta_0$-equicontinuous set are weak${}^*$ compact and sequentially weak${}^*$ compact.
\end{corollary}
\begin{proof}
The weak${}^*$ compactness of the closed convex circled hull of an equicontinuous set follows from \cite[18.5]{kelleynamioka}. Closure and closed convex hull are closed subsets of closed convex circled hull, from which the second statement follows. The $\beta_0$-equicontinuous sets are bounded in total variation, in particular, from below, so, the sequential statements are true, by Proposition~\ref{eqvcomp}.
\end{proof}

\subsubsection{Skorokhod's representation theorem}
\label{sec:skoro}

Jakubowski's fundamental observation was that Property~\ref{separatingfamily} yields a subsequential Skorokhod representation theorem.

\begin{proposition}
\label{prohoskoro2}
Let $(Q_n)_{n\in\mathbb{N}}$ be a sequence converging in the weak${}^*$ topology to $Q$ in $\P(\mathbb{D})$. Then there exists a subsequence $(Q_{n_k})_{k\in\mathbb{N}}$, a probability space $(\Omega,\F,P)$ and $\mathbb{D}$-valued random variables $(Y_k)_{k\in\mathbb{N}}$ and $Y$ on $(\Omega,\F,P)$ such that $Q_{n_k}=P\circ(Y_k)^{-1}$, $k\in\mathbb{N}$, $Q=P\circ Y^{-1}$ and
\begin{equation}
\label{eq:as}
f(Y_k(\omega)){\rightarrow} f(Y(\omega)),\ \forall\omega\in\Omega,\ \forall f\in\mathbb{C}_b(\mathbb{D}).
\end{equation}
\end{proposition}
\begin{proof}
The Euclidean space endowed with its usual inner product is a Hilbert space, so, by \cite[Theorem~1]{jakubowski2}, the existence of an a.s. convergent subsequence follows from Property~\ref{separatingfamily}. The convergence in \cite[Theorem~1]{jakubowski2} is the pointwise convergence in the topology of the underlying space, which, for a sequence in a completely regular space, is equivalent to the convergence~\eqref{eq:as}; cf. \eqref{convergence}. Moreover, modifying the $\mathbb{D}$-valued random variables $Y_k$ and $Y$, given by \cite[Theorem~1]{jakubowski2}, on a set of measure zero does not affect on their weak${}^*$ convergence, so, their almost sure convergence can be strengthened to the pointwise convergence.
\end{proof}

In particular, by Proposition~\ref{prohoskoro2}, every element of $\P(\mathbb{D})$ can be regarded as a law of some $\mathbb{D}$-valued random variable. Complementarily, any such random variable induces a probability measure on $\mathbb{D}$.

\subsection{Stability and tightness}
\label{sec:tightstab}

In this section, we cover the required stability and tightness results. We present the required multi-dimensional infinite horizon extensions of the stability results of Meyer and Zheng \cite{meyerzheng}, Jakubowski, M{\' e}min and Pages \cite{jakubowski6}, and Lowther \cite{lowther2009}, for the right-continuous version of the raw canonical filtration.

\subsubsection{Stability}
\label{sec:stab}

\emph{Under Assumption~\ref{ass:main}, it suffices to establish the required stability results for the Meyer-Zheng topology}; see Appendix~\ref{sec:mz}. The required stability results are classical and thoroughly studied in the aforementioned works \cite{meyerzheng}, \cite{jakubowski6} and \cite{jakubowski} for scalar-valued processes. We demonstrate that, after some slight modifications, they are true in the present setting. We utilize the following multi-dimensional extension of \cite[Threorem~5]{meyerzheng}, provided by Jakubowski's subsequential Skorokhod's representation theorem.

\begin{lemma} (Jakubowski)
\label{findim}
If $(Q_n)_{n\in\mathbb{N}}$ is a sequence converging in the weak${}^*$ topology to $Q$ in $\P(\mathbb{D}(I;\mathbb{R}^d))$, then there exists a subsequence $(Q_{n_k})_{k\in\mathbb{N}}$ and a set $L\subset I$ of full Lebesgue measure such that $T\in L$, if $I=[0,T]$, and
\begin{equation}
\label{findimeq}
Q_n\circ X^{-1}_F\rightarrow_{w^*} Q\circ X^{-1}_F,\text{ as }n\rightarrow\infty,
\end{equation}
for every finite subset $F$ of $L$. In particular, there exists a (countable) dense set $D\subset I$ such that $T\in D$, if $I=[0,T]$, and \eqref{findimeq} is true, for every finite subset $F$ of $D$.
\end{lemma}
\begin{proof}

By Proposition~\ref{prohoskoro2}, we find a subsequence $(Q_{n_k})$, $k\in\mathbb{N}$, and $\mathbb{D}$-valued random variables $(Y_k)_{k\in\mathbb{N}}$ and $Y$ on some $(\Omega,\F,P)$ such that $Q_{n_k}=P\circ Y_k^{-1}$, $k\in\mathbb{N}$, $Q=P\circ Y^{-1}$ and $Y_k(\omega)\rightarrow_{MZ} Y(\omega)$, for every $\omega\in\Omega$, as $k\rightarrow\infty$; since the topology $MZ$ is metrizable, \eqref{eq:as} is equivalent to $\rightarrow_{MZ}$. By Lemma~\ref{lem:lebesgue}, there exists a subsequence $Y_{k_m}(\omega)$, $m\in\mathbb{N}$, and a set $L$ of full Lebesgue such that $T\in L$, if $I=[0,T]$, and $Y_{k_m,t}(\omega)\rightarrow Y_t(\omega)$, for every $(t,\omega)\in L\times \Omega$, as $m\rightarrow\infty$. Hence, the finite dimensional distributions of the process $(Y_{k_m,t})_{t\in L}$ converge to those of $(Y_t)_{t\in L}$. The complement of the set $L$ is a $\lambda$-null set. Thus, the set $L$ contains a (countable) dense set $D$ such that $T\in D$, for $I=[0,T]$; cf. Definition~\ref{def:mzinf}.
\end{proof}

In Proposition~\ref{pro:semilimit}, we show that the required part of \cite[Theorem~2.1]{jakubowski6}, which is an extension of \cite[Theorem~2]{stricker2} for a right-continuous canonical filtration, is true on a multi-dimensional Skorokhod space.

\begin{proposition}
\label{pro:semilimit}
Let $(Q_n)_{n\in\mathbb{N}}$ be a sequence of semimartingale measures satisfying the condition \eqref{ut} and converging in the weak${}^*$ topology to $Q$. Then, the weak${}^*$ limit $Q$ is a semimartingale measure.
\end{proposition}
\begin{proof}
The proof is essentially a combination of \cite[Lemma~1.1~and~1.3]{jakubowski6}. By Lemma~\ref{findim}, there exists a subsequence $(Q_{n_k})_{k\in\mathbb{N}}$ and a countable dense set $D\subset I$ such that $T\in D$, if $I=[0,T]$, and $(Q_{n_k})_{k\in\mathbb{N}}$ converges to $Q$ in finite dimensional distributions on the set $D$. For every finite collection $t_1<\dots<t_j$ \emph{in} $D$, let $\A_{t_1,\dots,t_j}$ denote the family of continuity sets of the marginal law of $Q$ on $t_1<\dots<t_j$, i.e., $\A_{t_1,\dots,t_j}$ consists of Borel sets $B\in\otimes_{i\leq j}\B(\mathbb{R}^{d})$ for which $Q\circ X^{-1}_{t_1,\dots,t_j}[\partial B]=0$, where $\partial B$ denotes the (Euclidean) topological boundary of $B$ on $\mathbb{R}^{d\times j}$. Following \cite{jakubowski6}, we introduce an auxiliary class $\J(D)$ of integrands, determined by the weak${}^*$ limit $Q$ and the dense set $D$, that are of the form
\[
J=\sum_{i=1}^d\sum_{k=1}^{n(i)} J_{t^i_{k-1}}\one_{(t^i_{k-1},t^i_k]},\ n(i)\in\mathbb{N},\ i\leq d,
\]
where every $t^i_0<t^i_1<\dots<t^i_{n(i)}$ is a finite collection of elements of $D$ and every $J_{t^i_{k-1}}$ is a finite linear combination of indicator functions of the continuity sets of the marginal law of $Q$ on $s_1<\dots<s_j\leq t^i_{k-1}$, $s_1,\dots,s_j\in D$, embedded on $\mathbb{D}(I;\mathbb{R}^d)$ and bounded by $1$ in absolute value, i.e., $|J_{t^i_{k-1}}|\leq 1$ and each $J_{t^i_{k-1}}$ is of the form
\[
J_{t^i_{k-1}}=\sum_{\ell=1}^{p}\alpha^{\ell}\one_{A^{\ell}\circ X^{-1}_{s_1,\dots,s_{j}}},\ s_{j}\leq t^i_{k-1},\ \alpha^{\ell}\in\mathbb{R},\ A^{\ell}\in\A_{s_1,\dots,s_{j}},
\]
for some elements $s_1<\dots<s_{j}$ of $D$ and $j$ and $p$ finite, for every $i\leq d$, for every $k\leq n$. Now, since $(Q_{n_k})_{k\in\mathbb{N}}$ is converging to $Q$ in finite dimensional distributions on the set $D$, by the vectorial Portmanteau's lemma, see e.g. \cite[Lemma~2.2]{vaart}, we have
\begin{equation}
\label{eq:portbound}
\begin{split}
Q[|(J\bullet X)_t|> c]&=Q\circ X_D^{-1}[|(J\bullet X)_t\circ X_D|> c]\\
&\leq\liminf_{k\rightarrow\infty}Q_{n_k}\circ X_D^{-1}[|(J\bullet X)_t\circ X_D|> c]\\
&=\liminf_{k\rightarrow\infty}Q_{n_k}[|(J\bullet X)_t|> c],\ c>0,\ t\in I,
\end{split}
\end{equation}
where $J\in\J(D)\subset \E(Q^{n_k})$, for all $k\in\mathbb{N}$. Due to the condition~\eqref{ut}, for every $t\in I$, the term on the last line, in \eqref{eq:portbound}, tends to zero, uniformly over $\J(D)$, as $c\rightarrow\infty$, i.e., the family $\{(J\bullet X)_t:J\in\J(D)\}$, is $Q$-tight, i.e, bounded in probability $Q$, for every $t\in I$. The topology of the convergence in probability is metrizable, so, for every $t\in I$, the sets, that are contained in the (sequential) closure of $\{(J\bullet X)_t:J\in\J(D)\}$, are bounded, which, in particular, entails that the set $\{(H\bullet X)_t:H\in\E(Q)\}$ is bounded in probability $Q$. Indeed, we will show this, by adapting a sequence of approximation arguments from \cite{jakubowski6}. First, since $D$ is dense in $I$ and contains $T$, for every $t_0< t_1<\cdots< t_n$ in $I$, $n\in\mathbb{N}$, there exists $t_0^k\leq t_1^k\leq\cdots\leq t_n^k$ in $D$, $k\in\mathbb{N}$, such that $t_j<t_j^k$, for every $j\leq n$, for every $k\geq 1$, and $t_j\downarrow t_j^k$, for every $j\leq n$, as $k\rightarrow\infty$; we allow $t_j^k=T$, if $t_j=T$. Since $d$ and $n$ are finite, by the right-continuity of $X=(X^1,\dots,X^d)$, we have
\begin{equation}
\label{eq:denseright}
X^i_{t_j^k}\rightarrow X^i_{t_j}\text{, uniformly over }i=1,\dots,d\text{ and }j=1,\dots,n,\text{ as }k\rightarrow\infty.
\end{equation}
Secondly, for every $i\leq d$, for every $j\leq n$, for every $t_j<T$, any $\F_{t_j}$-measurable $|H^i_{t_j}|\leq 1$ is $\F^o_{t_j^k-}$-measurable, for all $k\geq 1$, and can therefore be expressed as an uniform limit of \emph{simple} $\F^o_{t_j^k-}$-measurable functions bounded by $1$ in absolute value, for every $k\geq 1$, i.e., for every $i\leq d$, for every $j\leq n$, for every $k\geq 1$, there exist functions $|S^{i,\ell}_{t^k_j}|\leq 1$, $\ell\in\mathbb{N}$, such that each $S^{i,\ell}_{t^k_j}$ is of the form
\begin{equation}
\label{eq:simple0}
S^{i,\ell}_{t^k_j}=\sum_{h=1}^{q(\ell)} \beta^{h,\ell}_{i,j,k}\one_{F^{h,\ell}_{i,j,k}},\ \beta^{h,\ell}_{i,j,k}\in\mathbb{R},\ F^{h,\ell}_{i,j,k}\in \F^o_{t_j^k-},\ 1\leq q(\ell)<\infty,
\end{equation}
and we have
\begin{equation}
\label{eq:simple}
\|H^i_{t^k_j}-S^{i,\ell}_{t^k_j}\|_\infty\rightarrow 0\text{ as }\ell\rightarrow\infty.
\end{equation}
Further, since each $\A_{t_1,\dots,t_j}$ is an algebra generating $\otimes_{i\leq j}\B(\mathbb{R}^{d})$ on $\mathbb{R}^{d\times j}$, and the finite unions of the cylindrical sets $X^{-1}_{t_1,\dots,t_j}(\otimes_{i\leq j}\B\left(\mathbb{R}^{d})\right)$ form an algebra generating the canonical $\sigma$-algebra on $\mathbb{D}(I;\mathbb{R}^d)$, for every $0<t\in I$, the family
\[
\A^o_{t-}=\left\{\bigcup_{k=1}^{n}X^{-1}_{t^k_1,\dots,t^k_{j(k)}}(A_k):A_k\in\A_{t^k_1,\dots,t^k_{j(k)}},\ t^k_1<\cdots<t^k_{j(k)}<t,\  j(k),n\in\mathbb{N}  \right\}
\]
is an algebra generating $\F^o_{t-}=\sigma(X_u:u<t)$ on $\mathbb{D}(I;\mathbb{R}^d)$; cf. Corollary~\ref{minusfilt}. Thus, for every $F^{h,\ell}_{i,j,k}\in \F^o_{t_j^k-}$ in \eqref{eq:simple0}, there exists a sequence $(A^{h,\ell,m}_{i,j,k})_{m\in\mathbb{N}}$ in $\A^o_{t_j^k-}$ such that
\begin{equation}
\label{eq:algebra}
\one_{A^{h,\ell,m}_{i,j,k}}\rightarrow_{Q}\one_{F^{h,\ell}_{i,j,k}},\text{ as }m\rightarrow\infty.
\end{equation}
Finally, by combining the approximations \eqref{eq:denseright} and \eqref{eq:simple}, and in \eqref{eq:simple}, invoking the approximation \eqref{eq:algebra} in the sums \eqref{eq:simple0}, we conclude that, for every $H\in\E(Q)$, there exists a sequence $(J_n)_{n\in\mathbb{N}}$, $n=n(k,\ell,m)$, of elements in $\J(Q)$ such that, for every $t\in I$, we have
\[
(J_n\bullet X)_t=\sum_{i=1}^d(J^i_n\bullet X^i)_t\rightarrow_{Q}\sum_{i=1}^d(H^i\bullet X^i)_t=(H\bullet X)_t,\text{ as }k\wedge \ell\wedge m\rightarrow\infty.
\]
Thus, the family of simple integrals $\{(H\bullet X)_t:H\in\E(Q)\}$ is contained in the closure of $\{(J\bullet X)_t:J\in\J(D)\}$, for every $t\in I$, and, by \eqref{eq:portbound}, the weak${}^*$ limit $Q$ is an $(\F_t)_{t\in I}$-semimartingale measure; and, consequently, an $(\F^o_t)_{t\in I}$-semimartingale measure; see e.g. \cite[Theorem~II.4]{protter}.

\end{proof}

The following Proposition~\ref{pro:quasilimit} is essentially \cite[Theorem~4]{meyerzheng}.

\begin{proposition}
\label{pro:quasilimit}
Let $(Q_n)_{n\in\mathbb{N}}$ be a sequence of quasimartingale measures satisfying the condition \eqref{ub} and converging in the weak${}^*$ topology to $Q$. Then, the weak${}^*$ limit $Q$ is a quasimartingale measure.
\end{proposition}
\begin{proof}
Let $i\leq d$ be fixed. We adapt the proof of \cite[Theorem~4]{meyerzheng} and show that the coordinate process $X^i$ is a quasimartingale under $Q$ on $\mathbb{D}(I;\mathbb{R}^d)$. We have
\[
E_{Q}\left[\frac{1}{\varepsilon}\int_0^\varepsilon|X^i_{(t+u)\wedge T}|du\right]\leq \liminf_{n\rightarrow\infty}E_{Q_n}\left[\frac{1}{\varepsilon}\int_0^\varepsilon|X^i_{(t+u)\wedge T}|du\right]\leq b^i,
\]
for every $t\in I$, for every $\varepsilon>0$, where $b^i:=\liminf_{n\rightarrow\infty}\sup_{t\in I}E_{Q_n}[|X^i_t|]<\infty$, by the condition \eqref{ub}. Thus, by Fatou's lemma, we get
\begin{equation}
\label{eq:intbound}
E_Q[|X^i_t|]\leq\liminf_{\varepsilon\rightarrow 0}E_{Q}\left[\frac{1}{\varepsilon}\int_0^\varepsilon |X^i_{(s+u)\wedge T}|du\right]<\infty,
\end{equation}
for every $t\in I$. The truncated coordinate process, that is,
\[
X^{i,c}:=(-c)\vee (X^i\wedge c)=(-c)\vee X^i-(X^i-c)^+,\ c>0,
\]
is a difference of two convex $1$-Lipschitz functions of $X^i$, for every $c>0$, so, we have
\begin{equation}
\label{eq:yoeurp}
\text{Var}^{Q_n}_t(X^{i,c})\leq 4\text{Var}^{Q_n}_t(X^i),\ n\in\mathbb{N},\ t\in I;
\end{equation}
see e.g. \cite{stricker4}. Let $0=t_0<t_1<\dots<t_k=t$ and $|f_j|\leq 1$, $j<k$, be \emph{continuous} $\F^o_{t_j-}$-measurable functions. By \eqref{eq:yoeurp}, we have
\begin{equation}
\label{eq:start}
E_{Q_n}\left[\sum_{j=1}^kf_{j-1}(X)(X^{i,c}_{(u+t_{j})\wedge T}-X^{i,c}_{(u+t_{j-1})\wedge T})\right]\leq 4\text{Var}^{Q_n}_t(X^i),\ n\in\mathbb{N},
\end{equation}
so, by Fubini's theorem, for every $n\in\mathbb{N}$, for every $\varepsilon>0$, we get
\[
E_{Q_n}\left[\frac{1}{\varepsilon}\int_0^\varepsilon\left( \sum_{j=1}^k f_{j-1}(X)(X^{i,c}_{(u+t_{j})\wedge T}-X^{i,c}_{(u+t_{j-1})\wedge T})\right)du\right]\leq 4\text{Var}^{Q_n}_t(X^i).
\]
The mappings
\[
F^\varepsilon(X):=\frac{1}{\varepsilon}\int_0^\varepsilon\left( \sum_{j=1}^k f_{j-1}(X)\left(X^{i,c}_{(u+t_{j})\wedge T}-X^{i,c}_{(u+t_{j-1})\wedge T}\right)\right)du,\ \varepsilon>0,
\]
are lower semicontinuous and bounded from below, see \eqref{inmeas} and Lemma~\ref{lem:lsc}, so, we have
\begin{equation}
\label{disp}
E_{Q}\left[\frac{1}{\varepsilon}\int_0^\varepsilon\left( \sum_{j=1}^kf_{j-1}(X)\left(X^{i,c}_{(u+t_{j})\wedge T}-X^{i,c}_{(u+t_{j-1})\wedge T}\right)\right)du\right]\leq 4 v^i,\ \varepsilon>0,
\end{equation}
where $v^i:=\liminf_{n\rightarrow\infty}\sup_{t\in I}\text{Var}^{Q_n}_t(X^i)<\infty$, by the assumption \eqref{ub}. Due to \eqref{eq:intbound}, letting $\varepsilon\rightarrow 0$ and then $c\rightarrow\infty$ in \eqref{disp}, by the right-continuity and the monotone convergence, respectively, we get
\begin{equation}
\label{disp2}
E_{Q}\left[\sum_{j=1}^kf_{j-1}(X)\left(X^{i}_{t_{j}}-X^{i}_{t_{j-1}}\right)\right]\leq 4 v^i,
\end{equation}
for all $\F^o_{t_j-}$-measurable continuous functions $|f_j|\leq 1$, $j<k$. Furthermore, by choosing $f_j(X)=f(X_{t_j+u})$ preceding \eqref{eq:start}, for a continuous function $|f|\leq 1$ on $\mathbb{R}^d$, we conclude that the inequality \eqref{disp2} is true for a family of continuous functions that, for every $j<k$, generates the $\sigma$-algebra $\F^o_{t_j}$; cf. Corollary~\ref{minusfilt}. Thus, by the standard $L^1$-approximation via Lusin's theorem and Tietze's extension theorem, for any $\F^o_{t_j}$-measurable $|H_{t_j}|\leq 1$ in $L^\infty(Q)$, for every $j<k$, the exists a sequence $(f^n_j)_{n\in\mathbb{N}}$, $|f^n_j|\leq 1$, of functions satisfying \eqref{disp2} such that $f^n_j\rightarrow H_{t_j}$ in $L^1(Q)$, as $n\rightarrow\infty$; see e.g. \cite{feldman} and \cite[2.1.8.]{engelking}. Thus, the inequality \eqref{disp2} is true for all $\F^o_{t_j}$-measurable functions $|H_{t_j}|\leq 1$ in $L^\infty(Q)$, so, the process $X$ is an $(\F^o_{t})_{t\in I}$-quasimartingale on $(\mathbb{D}(I),\F_T,Q)$, see \cite[B.~App.~II~(3.5)]{dellacheriemeyer}, which, by Rao's decomposition theorem, is a necessary and sufficient condition for the process $X$ to be decomposable to a difference $X=Y-Z$ of two c{\` a}dl{\` a}g $(\F^o_{t})_{t\in I}$-supermartingales $Y$ and $Z$ on $(\mathbb{D}(I),\F_T,Q)$; see \cite[Theorem~8.13]{hewangyan}. On the other hand, by F{\" o}llmer's lemma, $Y$ and $Z$ are $(\F_{t})_{t\in I}$-supermartingales, see \cite[Theorem~2.46]{hewangyan}, so, by Rao's decomposition theorem, the process $X$ is an $(\F_{t})_{t\in I}$-quasimartingale on $(\mathbb{D}(I),\F_T,Q)$.

\end{proof}

The following Proposition~\ref{pro:superlimit} is essentially \cite[Theorem~11]{meyerzheng}.

\begin{proposition}
\label{pro:superlimit}
Let $(Q_n)_{n\in\mathbb{N}}$ be a sequence of supermartingale measures satisfying the condition \eqref{ui} and converging in the weak${}^*$ topology to $Q$. Then, the weak${}^*$ limit $Q$ is a supermartingale measure.
\end{proposition}
\begin{proof}
We adapt the proof of \cite[Theorem~11]{meyerzheng} and show that each coordinate process $X^i$, $i\leq d$, is a supermartingale under $Q$ on $\mathbb{D}(I;\mathbb{R}^d)$. We have $E_Q[|X_t|]<\infty$, for every $t\in I$; cf. \eqref{eq:intbound}. Moreover, by Lemma~\ref{findim}, there exists a subsequence $(Q_{n_k})_{k\in\mathbb{N}}$ and a countable dense set $D\subset I$ such that $T\in D$, if $I=[0,T]$, and $(Q_{n_k})_{k\in\mathbb{N}}$ converges to $Q$ in finite dimensional distributions on the set $D$. Let $X^{i,c}$ denote the coordinate process $X^i$ truncated \emph{from above} at $c>0$, i.e.,
\[
X^{i,c}:=X^i\wedge c,\ c>0.
\]
By the condition~\eqref{ui} and the fact that each $Q_{n_k}$ is a supermartingale measure for $X^{i,c}$, for every $c>0$, we have
\[
E_Q[f(X)(X^{i,c}_t-X^{i,c}_s)]\leq\liminf_{k\rightarrow\infty}E_{Q_{n_k}}[f(X)(X^{i,c}_t-X^{i,c}_s)]\leq 0,\ s<t,\ s,t\in D,
\]
where
\[
f(X):=f_1(X_{t_1})f_2(X_{t_2})\cdots f_n(X_{t_n}),\ t_j\in D,\ f_j\in\mathbb{C}_b(\mathbb{R}^d),\ j\leq n;
\]
see e.g. \cite[Theorem~2.20]{vaart}. Consequently, by Corollary~\ref{minusfilt}, we have
\begin{equation}
\label{eq:supprop}
E_Q[\one_F(X)(X^{i,c}_t-X^{i,c}_s)]\leq 0,\ c>0,
\end{equation}
for every $s<t$ in $D$ and $F\in\F^o_{s-}$. Letting $c\rightarrow\infty$ in \eqref{eq:supprop}, by the monotone convergence theorem, we get the same inequality for the coordinate process $X^i$. By F{\" o}llmer's lemma, the inequality extends immediately to the whole $I$, and further, for $F\in\F^o_{s+}$. Indeed, we have
\begin{equation}
\label{eq:minusplus}
E_Q[\one_F(X)(X^i_t-X^i_s)]=\lim_{n\rightarrow\infty}E_Q[\one_F(X)(X^i_t-X^i_{s+1/n})]\leq 0,
\end{equation}
for every $F\in\F^o_{s+}$; cf. \cite[Theorem~2.44]{hewangyan}.
\end{proof}

For the sake of completeness, we provide the following Proposition~\ref{pro:markov}; the assertion~(a) is the classical Kolmogorov's criteria for the almost sure (H{\" o}lder) continuity and the assertion~(b) is essentially \cite[Proposition~6]{hirsch}; see also \cite[Lemma~4.5]{lowther2009}.

\begin{proposition}\label{pro:markov} 
Let $(Q_n)_{n\in\mathbb{N}}$ be a sequence of probability measures converging in the weak${}^*$ topology to $Q$. Then, we have the following:
\begin{enumerate}[(a)]
\item If the sequence $(Q_n)_{n\in\mathbb{N}}$ is $C$-tight, then the limit $Q$ is $C$-tight.
\item If each $Q_n$ is Lipschitz-Markov, then the limit $Q$ is Lipschitz-Markov.
\end{enumerate}
\end{proposition}
\begin{proof}
(a) By \cite[Lemma~15.49]{hewangyan}, the $C$-tighness of the sequence $(Q_n)_{n\in\mathbb{N}}$ implies the convergence along a subsequence in the weak${}^*$ topology of the Skorokhod's $J^1$-topology, and a fortiori in any weaker topology, to a law of a continuous process, which is the limit $Q$; cf. Proposition~\ref{pro:hierarchy}.

(b) Let $s<t$ in $I$ be fixed and take a bounded Lipschitz continuous function $g:\mathbb{R}^d\rightarrow\mathbb{R}$ with a Lipschitz constant $L(g)\leq 1$. For each $n\in\mathbb{N}$, there exists bounded Lipschitz continuous function $f_n:\mathbb{R}^d\rightarrow\mathbb{R}$ such that
\begin{equation}
\label{eq:aux0}
f_n(X_s)=E_{Q^n}[g(X_t)\mid\F^o_s]\quad Q\textnormal{-a.s.}.
\end{equation}
Further, we can take it granted that $L(f_n)\leq 1$ and $\|f_n\|_\infty\leq \|g\|_\infty<\infty$, for all $n\in\mathbb{N}$. Thus, by the Arzel{\' a}-Ascoli theorem, there exists a subsequence $(f_{n_k})_{k\in\mathbb{N}}$ and a bounded Lipschitz continuous function $f:\mathbb{R}^d\rightarrow\mathbb{R}$ with $L(f)\leq 1$ such that $f_{n_k}$ converges to $f$ uniformly on compacta, as $k\rightarrow\infty$. Further, by Lemma~\ref{findim}, there exists a further subsequence of $(Q_{n_k})_{k\in\mathbb{N}}$, that again we denote by $(Q_{n_k})_{k\in\mathbb{N}}$, converging in the finite-dimensional distributions to $Q$ on a dense set $D\subset I$ such that $T\in D$, if $I=[0,T]$. Let $i\in\mathbb{N}$ and $0\leq s_1<s_2<\cdots< {s_i}\leq s<t$ in $D$ and take a bounded continuous compactly supported $\alpha:\mathbb{R}^d\rightarrow\mathbb{R}$ and a bounded continuous $\beta:\mathbb{R}^{d\times i}\rightarrow\mathbb{R}$. By \eqref{eq:aux0}, we have
\begin{equation}
\label{eq:aux}
E_{Q_{n_k}}[(f_{n_k}(X_s)-g(X_t))\alpha(X_s)\beta(X_{s_1},\dots,X_{s_i})]=0,\quad \forall k\in\mathbb{N}.
\end{equation}
Since $f_{n_k}$ converges uniformly to $f$ in the compact support of each $\alpha$ and $Q_{n_k}$ converges to $Q$ in the finite dimensional distributions on the set $D$, as $k\rightarrow\infty$, from \eqref{eq:aux}, by the vectorial Portmanteau's lemma, see e.g. \cite[Lemma~2.2]{vaart}, we get
\begin{equation}
\label{eq:aux1}
E_{Q}[(f(X_s)-g(X_t))\alpha(X_s)\beta(X_{s_1},\dots,X_{s_i})]=0.
\end{equation}
The equality \eqref{eq:aux1} holds for all bounded continuous $\alpha$ with compact support, which yields
\begin{equation}
\label{eq:aux2}
E_{Q}[(f(X_s)-g(X_t))\beta(X_{s_1},\dots,X_{s_i})]=0,
\end{equation}
and further, the equality \eqref{eq:aux2} holds for all bounded continuous $\beta$, which yields
\begin{equation}
\label{eq:aux3}
E_{Q}[(f(X_s)-g(X_t))h(X)]=0,
\end{equation}
for every bounded $\F^o_s$-measurable function $h$, for every $s<t$ in $D$. Since $D$ is a dense subset of $I$ and $s\mapsto f(X_s)$ is bounded and right-continuous on $I$, by the bounded convergence theorem, the equality \eqref{eq:aux3} extends to the whole $I$, and further, for every bounded $\F_s$-measurable function $h$; cf. \eqref{eq:minusplus}.
\end{proof}

\subsubsection{Tightness}
\label{sec:tightness}

We say that a family $\Q$ of probability measures on $\left(\mathbb{D}(I;\mathbb{R}^d),\F_T\right)$ satisfies \emph{Jakubowski's uniform tightness criteria}, if we have
\[
\lim_{c\rightarrow\infty}\sup_{Q\in\Q}Q[\|X^{i,t}\|_\infty>c]=0\text{ and }\lim_{c\rightarrow\infty}\sup_{Q\in\Q}Q[N^{a,b}(X^{i,t})>c]=0,\ \forall a<b,\label{us} \tag{US}
\]
for every finite $t\in I$, for every $i\leq d$, where $X^{i,t}$ denotes the coordinate process $X^i$ \emph{restricted} on $[0,t]$; cf. Corollary~\ref{finalcrit}. It was shown in \cite{jakubowski} that a family of probability measures on $\left(\mathbb{D}([0,T];\mathbb{R}),\F_T\right)$, $T<\infty$, satisfies the condition \eqref{us} if and only if it is uniformly $S$-tight. In particular, we have the following hierarchy, cf. \eqref{eq:hierarchy},
\begin{equation}
\label{eq:hierarchy2}
\eqref{ut}\implies\textnormal{(US)}\implies\textnormal{(US${}^*$)},
\end{equation}
where (US${}^*$) stands for the uniform tightness in the $S^*$-topology; see Section~\ref{snot}.

The second implication in \eqref{eq:hierarchy2} is immediate from the definition of the $S^*$-topology; see Proposition~\ref{sproperties}~(i). The first implication in \eqref{eq:hierarchy2} follows from Proposition~\ref{semitight}, that is essentially the result of Stricker \cite[Theorem~2]{stricker}, which states that a sequence satisfying the condition \eqref{ut} admits a convergent subsequence and the limit law is a law of a semimartingale. Analogous results were obtained for the $S$-topology by Jakubowski in \cite[Theorem~4.1]{jakubowski}; see also \cite[Proposition~3.1]{jakubowski}.

\begin{proposition}
\label{semitight}
A family of semimartingale measures satisfying the condition \eqref{ut} satisfies the condition \eqref{us}.
\end{proposition}
\begin{proof}
Let $X^{i,t}$ denote the coordinate processes $X^i$ \emph{restricted} on $[0,t]$, for $i\leq d$ and $t<\infty$. Following Stricker \cite[Theorem~2]{stricker2}, we define a family of stopping times
\[
\tau^{i,c}=\inf\{s\in I:|X^{i,t}_s|>c\},\ i\leq d,\ c>0,
\]
and each $\tau^{i,c}$ is approximated from right with the sequence of the \emph{simple} stopping times
\begin{equation}
\label{eq:approx}
\tau^{i,c}_{n}=\min\{m/n :m\in\mathbb{N},\ \tau^{i,c}\leq m/n\},\ n\in\mathbb{N}.
\end{equation}
Since we are assuming a right-continuous filtration $(\F_t)_{t\in I}$ and $X^i$ is right-continuous, the hitting times $\tau^{i,c}$, and consequently, their approximations $\tau^{i,c}_{n}$ are indeed stopping times. Moreover, since each $\tau^{i,c}_{n}$ takes only finitely many values on $[0,t]$, every process $|H^n|\leq 1$ of the form
\begin{equation}
\label{eq:simeq}
H^{n}=\one_{[0,\tau^{i,c}_n\wedge t]},\ i\leq d,\ c>0,\ n\in\mathbb{N},\ t\in I,
\end{equation}
is an elementary predictable integrand; see \eqref{eq:strats}. Now, due to the right-continuity of $X^i$, by the bounded convergence theorem, for every $Q\in\Q$, we have
\begin{equation}
\label{eq:tigeq}
Q\left[ \|X^{i,t}\|_\infty >c\right]= Q\left[|(H^{n}\bullet X^i)_t| >c\ \forall n\in\mathbb{N}\right],
\end{equation}
for every $t\in I$, for every $c>0$. By the condition \eqref{ut}, the left-hand side in \eqref{eq:tigeq} tends to $0$, uniformly over $Q\in\Q$, for every $i\leq d$, for every $t\in I$, as $c\rightarrow\infty$.  Similarly, for $a<b$, we define, recursively, for all $k\in\mathbb{N}_0$, the stopping times
\[
\sigma^{i,a}_k=\inf\{s>\tau^{i,b}_{k-1}:|X^{i,t}_s|<a\},\ \tau^{i,b}_k=\inf\{s>\sigma^{i,a}_{k}:|X^{i,t}_s|>b\},\ \sigma^{i,a}_0=\tau^{i,b}_0=0,
\]
and the respective decreasing sequences $(\sigma^{i,a}_{k,n})_{n\in\mathbb{N}}$ and $(\tau^{i,b}_{k,n})_{n\in\mathbb{N}}$ of approximative stopping times, taking only finitely many values on finite intervals; cf. \eqref{eq:approx}. The processes $|H^{m,n}|\leq 1$, $m,n\in\mathbb{N}$, of the form
\[
H^{m,n}=\sum_{k=1}^m\one_{(\sigma^{i,a}_{k,n}\wedge t,\tau^{i,b}_{k,n}\wedge t]},
\]
are finite linear combinations of processes of the form \eqref{eq:simeq}, so, each process $|H^{m,n}|\leq 1$ is an elementary predictable integrand. Moreover, we have
\begin{equation}
\label{eq:tigeq2}
Q\left[ N^{a,b}(X^{i,t})>c\right]\leq Q\left[\lim_{m\rightarrow\infty}|(H^{m,n}\bullet X^i)_t| >a^++c(b-a)\ \forall n\in\mathbb{N}\right].
\end{equation}
By the condition \eqref{ut}, for every $a<b$, the right-hand side of \eqref{eq:tigeq2} tends to zero, uniformly over $Q\in\Q$, as $c\rightarrow 0$; cf. \eqref{eq:portbound}-\eqref{eq:denseright}. Thus, by Corollary~\ref{finalcrit}, the family $\Q$ satisfies the condition \eqref{us}.

\end{proof}

\subsection{The proofs of the main results}
\label{sec:theproofs}

In this section we provide the proofs for the results of Section~\ref{sec:maine} that we omitted there. We begin by proving Theorem~\ref{thm:semi} by invoking the results of Section~\ref{sec:weakstar} in conjunction with the stability and tightness result on the semimartingale property of Section~\ref{sec:tightstab}. Then, the rest of the results of Section~\ref{sec:maine} follow from the respective stability results of Section~\ref{sec:tightstab}.

\

\noindent
{\em{Proof of Theorem~\ref{thm:semi}}.} The condition \eqref{ut} is stronger than the condition ($\textnormal{US}{}^*$); cf. \eqref{eq:hierarchy2}. So, by Proposition~\ref{prohoskoro}, the family $\S$ is $\beta_0$-equicontinuous. Thus, by Corollary~\ref{cor:equi}, the closure of $\S$ is compact and sequentially compact in the weak${}^*$ topology; see Proposition~\ref{eqvcomp}. By Corollary~\ref{cor:weakseq}, the closure of $\S$ coincides with the sequential closure of $\S$. It remains to show that, every element in the sequential closure $[\S]_{seq}$ is a semimartingale measure. This particular fact is the statement of Proposition~\ref{pro:semilimit}.

\

\noindent
{\em{Proof of Corollary~\ref{cor:quasi}}.} The condition \eqref{ub} is weaker than the condition \eqref{ut}; see \eqref{eq:hierarchy}. By Proposition~\ref{pro:quasilimit}, the class of quasimartingale measures is stable in the weak${}^*$ convergence under the property~\eqref{ub}. Thus, Corollary~\ref{cor:quasi} follows from Theorem~\ref{thm:semi}.

\

\noindent
{\em{Proof of Corollary~\ref{cor:super}}.} The condition \eqref{ui} is weaker than the condition \eqref{ub}; see \eqref{eq:hierarchy}. By Proposition~\ref{pro:superlimit}, the class of supermartingale measures is stable in the weak${}^*$ convergence under the condition~\eqref{ui}. Thus, Corollary~\ref{cor:super} follows from Corollary~\ref{cor:quasi}.

\

\noindent
{\em{Proof of Lemma~\ref{lem:cm}}.} Since every sequence in the set is $C$-tight, by Proposition~\ref{pro:markov}~(a), every limit point in the sequential closure is $C$-tight. Thus, Lemma~\ref{lem:cm}~(a) is true. Likewise, Lemma~\ref{lem:cm}~(b) is a direct consequence of the stability of the (Lipschitz) Markov property in the weak${}^*$ convergence; see Proposition~\ref{pro:markov}~(b).

\section{The weak $S$-topology}
\label{snot}

We introduce the notion of weak $S$-topology and study its properties and relation to other topologies on the Skorokhod space.

\subsection{Definition}
\label{sec:sreg}

A possibility of defining a completely regular (non-sequential) $S$-topology is discussed already by Jakubowski in \cite{jakubowski}. Indeed, see the page~18 in \cite{jakubowski}. We describe a general method for regularizing any given topology. Our approach is inspired by the seminal work of Alexandroff \cite{alexandroff}. Let $\X=(X,\tau)$ be an arbitrary topological space and $\V$ an arbitrary subbase for the Euclidean topology on $\mathbb{R}$, then the family
\begin{equation}
\label{regularization}
\{f^{-1}(V):f\in\mathbb{C}_b(\X),\ V\in\V \}
\end{equation}
is a subbase for a (unique) topology on $X$. Indeed, the topology generated by the subbase \eqref{regularization} on $X$ is independent of the choice of the subbase $\V$ on $\mathbb{R}$; see e.g. \cite[3.4]{jerison}. The topology is generated by the family of pseudometrics
\begin{equation}
\label{pseudo}
\{\rho_{f_1,f_2,\dots,f_k}:f_1,f_2,\dots,f_k\in\mathbb{C}_b(\X)\},
\end{equation}
where
\[
\rho_{f_1,f_2,\dots,f_k}(x,y):=\max\{|f_1(x)-f_1(y)|,|f_2(x)-f_2(y)|,\dots,|f_k(x)-f_k(y)|\}
\]
for $x,y\in X$, and thus, the convergence of a net $(x_\alpha)$ to an element $x$ in this topology is equivalent to that
\begin{equation}
\label{convergence}
f(x_\alpha)\rightarrow f(x),\ \forall f\in\mathbb{C}_b(\X);
\end{equation}
see e.g. \cite[Example~8.1.19]{engelking}. We remark that by replacing $\mathbb{C}_b(\X)$ with $\mathbb{C}(\X)$ in \eqref{regularization}, \eqref{pseudo} and \eqref{convergence} one obtains an equivalent characterization. Any of these characterizations is necessary and sufficient criterion for a topological space to be completely regular ($T_{3{^1/_2}}$); see e.g. \cite[3.4]{jerison}.

\begin{definition}
\label{regs}
\textnormal{
We will denote by $S^*$ the topology generated on the Skorokhod space by the family \eqref{regularization} of $S$-continuous functions, and call it \emph{weak $S$-topology}.}
\end{definition}

\begin{remark}
\label{dirac}
The convergence in the weak${}^*$ topology on the $\beta_0$-dual of $\mathbb{C}_b(\mathbb{D})$, cf. Lemma~\ref{lem:riesz}, traditionally called the "weak convergence" for sequences of probability measures, is equivalent to the convergence \eqref{convergence}, if the measures are Dirac measures; see \cite[Lemma~8.9.2.]{bogachev}.
\end{remark}

\begin{remark}
\label{rem:jak}
It should be emphasized that, if one could show that the $S$-topology is regular (or linear), then the $S$- and the weak $S$-topology would coincide. It was communicated to the author by Professor Jakubowski that the regularity of $S$-topology remains as an open question.
\end{remark}

\subsection{Relation to other topologies}
\label{sec:rela}

The definitions of Jakubowski's $\Sigma$-topology, Jakubowski's $S$-topology, the Meyer-Zheng topology ($MZ$) and Skorokhod's $J^1$-topology are given in Appendix~\ref{sec:othertopos}. The $S^*$-topology is related to these topologies as follows.

\begin{proposition}
\label{pro:hierarchy}
We have $MZ\subset S^*$, $\Sigma\subset S^*$ and $S^* \subset S\subset J^1$.
\end{proposition}
\begin{proof}
The functions in \eqref{inmeas} and \eqref{interm}, that generate the topology $MZ$, are $S^*$-continuous; see Example~\ref{contexamples}. Moreover, the topology $MZ$ is metrizable, in particular, sequential and completely regular. Thus, by Example~\ref{contexamples}, the first inclusion $MZ\subset S^*$ is true; cf. \eqref{convergence}. The topology $\Sigma$ is a completely regular topology weaker than $S$; see \cite[Remark~3.8]{jakubowski4} and \cite[Theorem~1.6.5.]{bogachevtvs}. Since the topology $S^*$ is the strongest completely regular weaker than $S$, we have $\Sigma \subset S^*\subset S$. The final inclusion $S\subset J^1$ is proved in \cite{jakubowski} for a finite compact interval, and extends immediately for the infinite interval due to \eqref{eq:jinf}; cf. \eqref{eq:sinf}.
\end{proof}

The Skorokhod space endowed with the $J^1$-topology is a Polish space, so, the space is a Lusin space for any topology that is weaker than the $J^1$-topology. The following Theorem~\ref{thm:hierarchy} states that the $S^*$-topology, which is the strongest (completely) regular topology that is weaker than the $S$-topology, is the strongest (completely) regular Souslin topology on the Skorokhod space for which the sets \eqref{eq:cond} are compact, and consequently, Jakubowski's uniform tightness criteria \eqref{us} is a sufficient tightness criteria; cf. Section~\ref{sec:tightness}.

\begin{theorem}
\label{thm:hierarchy}
Let $T$ be a completely regular Souslin topology on the Skorokhod space, comparable to $S$, and $\K(T)=\K(S)$. Then
\[
T\subset S.
\]
\end{theorem}
\begin{proof}
Assume that $S\subset T$ and let $T_s$ denote the sequential topology generated by $T$. Since the compact sets of a completely regular Souslin space are metrizable, we have $\K(T)\subset\K(T_s)$; see e.g. \cite[p.~218]{bogachev}. Consequently, we have
\begin{equation}
\label{eq:a}
\K(S)=\K(T)=\K(T_s),
\end{equation}
where
\begin{equation}
\label{eq:b}
S\subset T\subset T_s
\end{equation}
and $S$ and $T_s$ are sequential; see Appendix~\ref{sec:sdef}. By \cite[Theorem~3.3.20.]{engelking}, the Skorokhod space is a (Hausdorff) $k$-space for $S$ and $T_s$, so, by \eqref{eq:a} and \eqref{eq:b}, we have $S=T$.
\end{proof}

\begin{remark}\label{rem:strongest}
For the Riesz representation theorem given in Lemma~\ref{lem:riesz} and the required auxiliary results given in Section~\ref{sec:basicresults} it is necessary that the underlying topological space is completely regular and Souslin. Among all such topologies, the topology $S^*$ is the strongest one that is weaker than $S$. Recall that, in addition to Lemma~\ref{lem:riesz} and results of Section~\ref{sec:basicresults}, the proof of the main result Theorem~\ref{thm:semi} goes through the property \eqref{us}, while the underlying relative compactness criteria \eqref{eq:cond} that gives arise for the tightness \eqref{us} is both necessary and sufficient for $S$; see also Remark~\ref{rem:char}. On the other hand, as shown in Theorem~\ref{thm:hierarchy} above, any topology with the previously cited properties and the appropriate compact closed sets \eqref{eq:cond} cannot be strictly stronger than $S$.
\end{remark}

\subsection{Properties}
\label{sec:sprop}

Recall Property~\ref{separatingfamily} from Section~\ref{sec:basicresults}. A topological space satisfying Property~\ref{separatingfamily} is submetrizable, from which various useful properties follow; see \eqref{eq:metric} and \cite{jakubowski5}. In fact, all key properties of the $S$-topology follow immediately from Property~\ref{separatingfamily}, and Property~\ref{separatingfamily} is preserved in the regularization \eqref{regularization}.

\begin{proposition}
\label{sproperties}
The $S$-topology has the following properties:
\begin{enumerate}[(a)]
\item $S$ is Hausdorff,
\item Each $K\in\K(S)$ is metrizable,
\item A set is compact if and only if it is sequentially compact,
\item The Borel $\sigma$-algebra $\B(S)$ and the canonical $\sigma$-algebra coincide,
\item The Skorokhod space endowed with $S$ is a Lusin space.
\end{enumerate}
The $S^*$-topology has the properties (a)-(e) and additionally:
\begin{enumerate}[(a)]
\setcounter{enumi}{5}
\item The Skorokhod space endowed with $S^*$ is perfectly normal and paracompact,
\item The Borel $\sigma$-algebra $\B(S^*)$ and the Baire $\sigma$-algebra $\B a(S^*)$ coincide,
\item $\mathbb{C}(S)=\mathbb{C}(S^*)$,
\item $\K(S)=\K(S^*)$.
\end{enumerate}
\end{proposition}
\begin{proof}
The properties (a), (b) and (c) follow immediately from the fact that the (weak) $S$-topology satisfies Property~\ref{separatingfamily}; see \cite[pages~10-11]{jakubowski5}. Indeed, the mappings
\begin{equation}
\label{eq:separ}
\omega\mapsto\frac{1}{r}\int_q^{q+r}\omega^i(t)dt\text{ and }\omega\mapsto\omega^i(T),\text{ for }I=[0,T],
\end{equation}
where $q$ and $q+r$ run over the rationals in $I$ and $i$ over the spatial dimensions $1,\dots,d$, constitute a countable family of continuous functions that separates the Skorokhod space; cf. Example~\ref{contexamples}.

(d) We prove the claim for $I=[0,T]$. The proof is completely similar for $I=[0,\infty)$. Fix a coordinate $i\leq d$. For all $0\leq t<T$, we have
\[
\omega^i(t)=\lim_{\delta\rightarrow 0}\frac{1}{\delta}\int_t^{t+\delta}\omega^i(u)du,
\]
i.e., the mapping $\omega\mapsto\omega^i(t)$ is a limit of elements of $\mathbb{C}(S)$, for every $t$ in $I$, while for $t=T$, the mapping $\omega\mapsto\omega^i(t)$ is an element of $\mathbb{C}(S)$. Consequently, we have $\sigma(X_u:u\in I)\subset\B(S^*)$ and since $S^*$ is weaker than $S$, we have $\B(S^*)\subset\B(S)$. On the other hand, by Proposition~\ref{pro:hierarchy}, $S$ is weaker than $J^1$, so, we have $\B(S)\subset\B(J^1)$, where $\B(J^1)=\sigma(X_u:u\in I)$. Thus, $\B(S^*)=\B(S)=\sigma(X_u:u\in I)$. By Proposition~\ref{pro:hierarchy}, we have $S^*\subset S\subset J^1$ and the Skorokhod space endowed with $J^1$ is a Polish space, so, the Skorokhod space endowed with $S$ or $S^*$ is a Lusin space. Thus, we have (e).

The Skorokhod space endowed with $S^*$ is a (completely) regular Souslin space. By the result of Fernique, every regular Souslin space is perfectly normal and paracompact; see \cite[Proposition~I.6.1]{fernique}. Thus, we have (f). Now, by (f), the Skorokhod space endowed with $S^*$ is perfectly normal, and consequently, by \cite[Proposition~6.3.4.]{bogachev}, we have $\B(S^*)=\B a(S^*)$, i.e., we have (g). The claim (h) follows directly from Definition~\ref{regs}.

To prove (i), we first observe that $\K(S)\subset\K(S^*)$, by Definition~\ref{regs}. To prove the converse inclusion we use Jakubowski's $\Sigma$-topology; see Appendix~\ref{sec:sigma}. By \cite[Remark~3.6]{jakubowski4}, we have $\Sigma\subset S$, so, we have $\mathbb{C}(\Sigma)\subset\mathbb{C}(S)=\mathbb{C}(S^*)$. Thus, we have that $\Sigma\subset S^*$, since the topology $\Sigma$ is completely regular. Indeed, topological vector spaces are completely regular; see e.g. \cite[Theorem~1.6.5.]{bogachevtvs}. Consequently, by \cite[Remark~3.8]{jakubowski4}, we get $\K(S^*)\subset\K(\Sigma)=\K(S)$. Thus, we have shown $\K(S)=\K(S^*)$.

\end{proof}

\begin{remark}
\label{rem:initial}
A countable product of regular Souslin spaces is a regular Souslin space. Thus, by the result of Fernique \cite[Proposition~I.6.1]{fernique}, the previous properties (after the obvious modifications) are inherited for (at most) countable products of $S^*$-topologies; cf. Section~\ref{sec:weakstar}.
\end{remark}

\subsection{Compact sets and continuous functions}
\label{sec:contandcomp}

In this section, we recall the compactness and continuity criteria for the $S$-topology from \cite{jakubowski} and \cite{jakubowski4}.

\subsubsection{Compactness criteria}

The \emph{necessity and sufficiency} of the condition \eqref{eq:cond} for the relative (sequential) compactness in the $S$-topology was proved in \cite{jakubowski}, for $I=[0,T]$, $T<\infty$, and the multi-dimensional infinite horizon extension was provided in \cite{jakubowski4}.

\begin{proposition}
\label{precompcrit} 
A subset $K$ of $\mathbb{D}([0,T];\mathbb{R})$, $T<\infty$, is relatively sequentially $S$-compact if and only if the following conditions are satisfied:
\begin{equation}
\label{eq:cond}
\begin{cases}
    &\sup_{\omega\in K}\|\omega\|_\infty<\infty,\\
    &\sup_{\omega\in K} N^{a,b}(\omega)<\infty,\ \forall a<b,\ a,b\in\mathbb{R}.
\end{cases}
\end{equation}
\end{proposition}

\begin{remark}\label{rem:char}
A right-continuous function $\omega:[0,T]\rightarrow\mathbb{R}$ is c{\`a}dl{\`a}g if and only if the following conditions are satisfied:
\[
\begin{cases}
    &\|\omega\|_\infty<\infty,\\
    & N^{a,b}(\omega)<\infty,\ \forall a<b,\ a,b\in\mathbb{R}.
\end{cases}
\]
\end{remark}

\begin{proposition}
\label{precompcrit2} 
A subset $K$ of $\mathbb{D}([0,\infty );\mathbb{R})$ is relatively sequentially $S$-compact if and only if the set $K$ restricted on $[0,t]$ satisfies the conditions \eqref{eq:cond} for every $0<t<\infty$.
\end{proposition}

We make the following observations.

\begin{enumerate}

\item For any two real numbers $a<b$ one can find rationals $r<q$ so that $a<r<q<b$, so, it is sufficient to let $a<b$ range rationals in Proposition~\ref{precompcrit}.

\item The mappings of $\omega$ in Proposition~\ref{precompcrit} are lower semicontinuous in the $S^*$-topology, so, their lower level sets are closed in the $S^*$-topology; cf. Example~\ref{lscexamples}.

\item A Cartesian product set in a multi-dimensional Skorokhod space is relatively sequentially $S$-compact if and only if each set in the product is relatively sequentially $S$-compact; cf. Definition~\ref{def:sconvergence}.

\item $S$-compact set are $S^*$-compact; cf. Proposition~\ref{sproperties}~(i).

\end{enumerate}

Combining the previous facts we obtain the following \emph{compactness} criteria.

\begin{corollary}\label{finalcrit} Let $K=K^1\times\cdots\times K^d$ be a Cartesian product set on the Skorokhod space $\mathbb{D}(I;\mathbb{R}^d)$ endowed with $S$ or $S^*$. Then the set $K$ is compact, if, for each $i\leq d$, there exists a (non-decreasing) function $C^{i}_{q,r}:I\rightarrow\mathbb{R}_+$, for all $q<r$ in $\mathbb{Q}$, and a (non-decreasing) function $M^{i}:I\rightarrow\mathbb{R}_+$ such that
\[
K^i:=\bigcap_{{q<r}}\{\omega^{i}:N^{q,r}\left([\omega^{i}]^t\right)\leq C^{i}_{q,r}(t)\text{ and }\|[\omega^{i}]^t\|_\infty\leq M^{i}(t)\ \forall t<\infty\},
\]
where the intersection is taken over all rationals $q<r$ and $[\omega^{i}]^t$ denotes the restriction of $\omega^{i}$ on $[0,t]$.
\end{corollary}

Remark that, for $I=[0,T]$, $T<\infty$, it suffices to consider constant $C^i_{q,r}$ and $M^i$ in Corollary~\ref{finalcrit}.

\subsubsection{Examples of (semi-)continuous functions}
\label{conti}

By Proposition~\ref{sproperties}~(h), $S^*$-continuous functions are precisely the $S$-continuous ones. In particular, we would like to emphasize that the evaluation mapping at $t$ is not continuous for any $t<T$; see \cite[p.11]{jakubowski}.

\begin{example}
\label{contexamples}
\begin{enumerate}[(a)]
\item The following mappings are $S^*$-continuous on $\mathbb{D}(I;\mathbb{R}^d)$
\[
\omega\mapsto\int_I G(t,\omega^i(t)) d\mu(t),\ i\leq d,
\]
whenever $G$ is measurable as a mapping of $(t,x)$, continuous as a mapping of $x$, for every $t\in I$, and such that
\begin{equation}
\label{eq:G}
\sup_{0\leq t\leq c}\sup_{|x|\leq c}|G(t,x)|<\infty,\ \forall c>0,
\end{equation}
and $\mu$ is a diffusive (an atomless) measure on $I$; see \cite[Corollary~2.11]{jakubowski}.

\item The mapping
\[
\omega\mapsto\omega(T)
\]
is $S^*$-continuous on $\mathbb{D}([0,T];\mathbb{R}^d)$; see \cite[Remark~2.4]{jakubowski}.

\end{enumerate}
\end{example}

By Proposition~\ref{pro:hierarchy}, the $S^*$-topology is stronger than the Meyer-Zheng topology ($MZ$), so, the uniform norm and the number of upcrossings of an interval $[a,b]$ are lower semicontinuous functions; see Lemma~\ref{lem:lsc}.

\begin{example}
\label{lscexamples}
The functions
\[
\omega\mapsto\|\omega\|_\infty\text{ and }\omega\mapsto N^{a,b}(\omega^i),\ a<b,\ i\leq d,
\]
are lower semicontinuous in the $S^*$-topology on $\mathbb{D}(I;\mathbb{R}^d)$.
\end{example}

\appendix\normalsize 
\section{Appendix}\label{sec:appendix}

The appendix collects the definitions of topologies and auxiliary results used in the main part of the article.

\subsection{Topologies on the Skorokhod space}
\label{sec:othertopos}

We recall the definitions of Jakubowski's $S$-topology and $\Sigma$-topology, the Meyer-Zheng pseudo-path topology and Skorokhod's $J^1$-metric topology. We define each topology separately on $\mathbb{D}([0,T];\mathbb{R}^d)$, for $T<\infty$, and $\mathbb{D}([0,\infty);\mathbb{R}^d)$. In particular, we use a formal definition of the Meyer-Zheng pseudo-path topology ($MZ$) that takes into account the fluctuations of the terminal value in the case of a finite time horizon. The space $\mathbb{D}([0,\infty];\mathbb{R}^d)$ is regarded as a product space $\mathbb{D}([0,\infty];\mathbb{R}^d)=\mathbb{D}([0,\infty);\mathbb{R}^d)\times\mathbb{R}^d$, where the space $\mathbb{R}^d$ is endowed with the Euclidean topology.

\subsubsection{The $S$-topology}
\label{sec:sdef}

Jakubowski's $S$-topology, introduced in \cite{jakubowski}, is a sequential topology. The following definition of the $S$-convergence on $\mathbb{D}([0,T];\mathbb{R})$ is taken from \cite{jakubowski}; the multi-dimensional version can be found in \cite{jakubowski4}.

\begin{definition}\label{def:sconvergence}\textnormal{ On $\mathbb{D}([0,T];\mathbb{R}^d)$, we write $\omega_n\rightarrow_{S}\omega_0$, if, for every $i\leq d$, for every $\varepsilon>0$, one can find $(\nu^{i,\varepsilon}_n)_{n\in\mathbb{N}_0}\subset\mathbb{V}([0,T])$ such that
\[
\|\omega^i_n-\nu^{i,\varepsilon}_n\|_\infty\leq\varepsilon,\ \forall n\in\mathbb{N}_0,\text{ and }\nu^{i,\varepsilon}_n\rightarrow_{w^*}\nu^{i,\varepsilon}_0,\text{ as }n\rightarrow\infty,
\]
where the convergence "$\rightarrow_{w^*}$" is in the weak${}^*$ topology on $\mathbb{V}([0,T])$, which can be identified with the Banach dual of $\mathbb{C}([0,T])$, under the uniform norm.}
\end{definition}

The following definition of the $S$-convergence on $\mathbb{D}([0,\infty);\mathbb{R}^d)$ is taken from \cite{jakubowski4}.

\begin{definition}\textnormal{  On $\mathbb{D}([0,\infty);\mathbb{R}^d)$ we write $\omega_n\rightarrow_{S}\omega_0$, if, for every $i\leq d$, one can find a sequence of positive real numbers $(T^r)_{r\in\mathbb{N}}$, increasing to $\infty$, such that
\begin{equation}
\label{eq:sinf}
[\omega^i_n]^{T^r}\rightarrow_{S}[\omega^i_0]^{T^r},\text{ for every }r\in\mathbb{N},
\end{equation}
where $[\omega^i]^{T^r}$ denotes the restriction of a path $\omega^i\in\mathbb{D}([0,\infty);\mathbb{R})$ on $\mathbb{D}([0,T^r];\mathbb{R})$.}
\end{definition}

A topological convergence is obtained by requiring that every subsequence admits a further $S$-convergent subsequence; see \cite[Theorem~6.3]{jakubowski4}. The following definition for the $S$-topology on the Skorokhod space $\mathbb{D}([0,T];\mathbb{R}^d)$ and $\mathbb{D}([0,\infty);\mathbb{R}^d)$ are taken from \cite{jakubowski} and \cite{jakubowski4}, respectively.

\begin{definition}\textnormal{ 
The $S$-topology is the topology generated on the Skorokhod space by the subsequential $S$-convergence.}
\end{definition}

\emph{The Skorokhod space endowed with the $S$-topology is known to be a Hausdorff ($T_2$) space and a stronger separation axiom is an open problem.} A weak separation axiom is a well-known issue for topologies defined via the subsequential convergence (KVPK recipe); see \cite[Appendix]{jakubowski4} for elaboration. The difficulties  encountered in establishing the regularity of the $S$-topology are explained in \cite[Rem.~3.12]{jakubowski}.

\subsubsection{The $\Sigma$-topology}\label{sec:sigma}

Jakubowski's $\Sigma$-topology was introduced in \cite{jakubowski4}. The Skorokhod space endowed with $\Sigma$ is a locally convex vector space. Following \cite{jakubowski4}, we start by defining an auxiliary mode of convergence $\rightarrow_\tau$ on the space of (bounded) continuous functions of finite variation $\mathbb{A}([0,T];\mathbb{R}):=\mathbb{C}([0,T];\mathbb{R})\cap\mathbb{V}([0,T];\mathbb{R})$. Namely, for a sequence $(A_n)_{n\in\mathbb{N}_0}\subset\mathbb{A}([0,T];\mathbb{R})$, we write $A_n\rightarrow_\tau A_0$, if
\[
\|A_n-A_0\|_\infty\rightarrow 0,\text{ as }n\rightarrow\infty,
\]
and
\[
\sup_{n\in\mathbb{N}_0}\|A_n\|_{\mathbb{V}}<\infty,
\]
where $\|\cdot\|_{\mathbb{V}}$ denotes the total variation norm.

\begin{definition}
The topology $\Sigma$ on $\mathbb{D}([0,T);\mathbb{R}^d)$ is the topology generated by the seminorms
\[
\rho^i_{\A}=\sup_{A\in\A}\left|\int_{[0,T)}\omega^i(u)dA(u)\right|,\ i\leq d,
\]
where $\A$ ranges over relatively $\tau$-compact subsets of $\mathbb{A}([0,T);\mathbb{R})$.
\end{definition}

\begin{definition}
The topology $\Sigma$ on $\mathbb{D}([0,T];\mathbb{R}^d)$ is the topology generated by the seminorm $\rho^i_T(\omega)=|\omega^i(T)|$, $i\leq d$, and the seminorms
\[
\rho^i_{\A}=\sup_{A\in\A}\left|\int_{[0,T]}\omega^i(u)dA(u)\right|,\ i\leq d,
\]
where $\A$ ranges over relatively $\tau$-compact subsets of $\mathbb{A}([0,T];\mathbb{R})$.
\end{definition}

The topology $\Sigma$ was defined on the Skorokhod space $\mathbb{D}([0,T];\mathbb{R})$ for $T=1$ in \cite{jakubowski4}. The following properties were shown to be true for $\Sigma$ on $\mathbb{D}([0,T];\mathbb{R})$.

\begin{proposition}\label{pro:sigmaproperties}
The $\Sigma$-topology has the following properties:
\begin{enumerate}[(i)]
\item The Skorokhod space endowed with $\Sigma$ is a locally convex vector space.
\item The topology $\Sigma$ is weaker than the topology $S$.
\item A set is $\Sigma$-compact if and only if it is $S$-compact.
\end{enumerate}
\end{proposition}

\begin{remark}\label{rem:sigmainfinite}
It was communicated to the author by Professor Jakubowski that the properties of Proposition~\ref{pro:sigmaproperties} remain true for the infinite horizon extension of the $\Sigma$-topology.
\end{remark}

\subsubsection{The Meyer-Zheng topology}
\label{sec:mz}

The Meyer-Zheng topology, introduced in \cite{meyerzheng}, is a relative topology, on the image measures on the graphs $(t,\omega(t))_{t\in[0,\infty]}$ of trajectories $(\omega(t))_{t\in[0,\infty]}$ under the measure $\lambda[dt]:=e^{-t}dt$ (called pseudo-paths), induced by the weak topology on probability laws on compactified space $[0,\infty]\times\overline{\mathbb{R}}$. We refer to the Meyer-Zheng topology formally ($MZ$) as the topology on the Skorokhod space $\mathbb{R}(I;\mathbb{R}^d)$ generated by the coordinatewise convergence in measure; see \eqref{inmeas}. The following definition is adapted from \cite[Lemma~1]{meyerzheng}, which states that, on $\mathbb{D}([0,\infty);\mathbb{R})$, the convergence in measure \eqref{inmeas} is indeed equivalent to the convergence in the pseudo-path topology.

\begin{definition}
\label{def:mzinf}\textnormal{ 
The topology $MZ$ on $\mathbb{D}(I;\mathbb{R}^d)$, where $I=[0,\infty)$, is the topology generated by the convergence:
\begin{equation}
\label{inmeas}
\int_I f(t,\omega^i_n(t))\lambda[dt]\rightarrow\int_I f(t,\omega^i(t))\lambda[dt],\ \forall f\in\mathbb{C}_b(I\times\mathbb{R}),\ \forall i\leq d,
\end{equation}
where $\lambda[dt]:=e^{-t}dt$.}
\end{definition}

On $\mathbb{D}([0,T];\mathbb{R}^d)$, we additionally require the convergence of the terminal value; \eqref{interm} below. Without this addition, the topology is not a Hausdorff topology on $\mathbb{D}([0,T];\mathbb{R}^d)$.

\begin{definition}\textnormal{ 
The topology $MZ$ on $\mathbb{D}(I;\mathbb{R}^d)$, where $I=[0,T]$, is the topology generated by the convergence \eqref{inmeas} in conjunction with the convergence:
\begin{equation}
\label{interm}
\omega_n(T)\rightarrow\omega(T).
\end{equation}
}
\end{definition}

The key lemma of Meyer and Zheng \cite[Lemma~1]{meyerzheng}, extends to $I=[0,T]$, for $T$ finite, and $d>1$ via a simple iterative argument; cf. Section~\ref{sec:stab}.

\begin{lemma}
\label{lem:lebesgue}
Let $(\omega_n)_{n\in\mathbb{N}}$ and $\omega$ be paths in $\mathbb{D}(I;\mathbb{R}^d)$ such that $\omega_n\rightarrow_{MZ}\omega$. Then $\omega^i_n\rightarrow_\lambda\omega^i$, for every $i\leq d$. Moreover, there exists a subsequence $(\omega_{n_k})$ and a set $L\subset I$ of full Lebesgue measure such that $T\in L$, if $I=[0,T]$, and $\omega^i_{n_k}(t)\rightarrow\omega^i(t)$, for every $i\leq d$, for every $t\in L$. In particular, there exists a (countable) dense set $D\subset I$ such that $T\in D$, if $I=[0,T]$, and $\omega^i_{n_k}(t)\rightarrow\omega^i(t)$, for every $i\leq d$, for every $t\in D$.
\end{lemma}
\begin{proof}
Let $\omega_n\rightarrow_{MZ}\omega$. By the definition \eqref{inmeas}, we have
\[
\int_I f(t,\omega^i_n(t))\lambda[dt]\rightarrow\int_I f(t,\omega^i(t))\lambda[dt],\ \forall f\in\mathbb{C}_b(I\times\mathbb{R}^d),\ \forall i\leq  d,
\]
where the measure $\lambda[dt]=e^{-t}dt$ is equivalent to the Lebesgue measure on $I$. Taking $f(t,x):=\alpha(t)\arctan(x)$, $\alpha\in\mathbb{C}_b(I)$, we deduce that $u^i_n:=\arctan(\omega^i_n)$ converges weakly to $u^i:=\arctan(\omega^i)$ in $L^2(\lambda)$, for every $i\leq d$. Further, taking $f(t,x):=\alpha(t)\arctan^2(x)$, $\alpha\in\mathbb{C}_b(I)$, we deduce that $u^i_n$ converges strongly to $u^i$ in $L^2(\lambda)$, and consequently, $\omega^i_n$ converges in measure $\lambda$ to $\omega^i$ in $I$, i.e., $\omega^i_n\rightarrow_\lambda\omega^i$, for every $i\leq d$. Thus, for $i=1$, there exists a subsequence $(\omega_{n_\ell})_{\ell\in\mathbb{N}}=(\omega^1_{n_\ell},\dots,\omega^d_{n_\ell})_{\ell\in\mathbb{N}}$ of $(\omega_n)_{n\in\mathbb{N}}$ such that 
\begin{equation}
\label{eq:iter}
\omega^1_{n_\ell}(t)\rightarrow\omega^1(t),
\end{equation}
for every $t$ in some set $L_1$ of full Lebesgue measure. By the bounded convergence theorem, we have
\[
\int_I f(t,\omega^i_{n_\ell}(t))\lambda[dt]\rightarrow\int_I f(t,\omega^i(t))\lambda[dt],\ \forall f\in\mathbb{C}_b(I\times\mathbb{R}^d),\ \forall i\leq d.
\]
Now, by replacing $i=1$ with $i=2$ and $(\omega_n)_{n\in\mathbb{N}}$ with $(\omega_{n_\ell})_{\ell\in\mathbb{N}}$ preceding \eqref{eq:iter}, we obtain a set $L_2$ of full Lebesgue measure and a further subsequence $(\omega_{{n_\ell}_m})_{m\in\mathbb{N}}=(\omega^1_{n_{\ell_m}},\dots,\omega^d_{n_{\ell_m}})_{m\in\mathbb{N}}$ of $(\omega_n)_{n\in\mathbb{N}}$ such that $\omega^2_{n_{\ell_m}}(t)\rightarrow\omega^2(t)$, for every $t\in L_2$. We have
\[
\omega^1_{n_{\ell_m}}(t)\rightarrow\omega^1(t)\text{ and }\omega^2_{n_{\ell_m}}(t)\rightarrow\omega^2(t),
\]
for every $t\in L_1\cap L_2$, where the set $L_1\cap L_2$ is of full Lebesgue measure. By repeating the argument $d-2$ more times, we obtain a set $L:=L_1\cap L_2\cap\dots\cap L_d$ and a subsequence $(\omega_{n_k})_{k\in\mathbb{N}}=(\omega^1_{n_{k}},\dots,\omega^d_{n_{k}})_{k\in\mathbb{N}}$ such that
\[
\omega^i_{n_{k}}(t)\rightarrow\omega^i(t),\ \forall i\leq d,
\]
for every $t\in L$, where the set $L$ is of full Lebesgue measure. Moreover, by \eqref{interm}, for $I=[0,T]$, we have $\omega_n(T)\rightarrow\omega(T)$, so, the set $L$ can be chosen to contain the terminal time $T$. The complement of $L$ is a $\lambda$-null set, so, the set $L$ contains a (countable) dense set $D$ such that $T\in D$, if $I=[0,T]$.
\end{proof}

\begin{corollary}
\label{minusfilt}
We have $\F^o_{t-}:=\sigma(\omega(s):s\in [0,t))=\sigma(\omega(s):s\in D\cap[0,t))$, for any countable dense subset $D$ of $[0,t)$, for every $t\leq T$. Moreover, we have
\[
\F^o_{t}=\sigma(\G^o_t,\omega(t)),\ t\leq T,
\]
where $\G^o_t$ denotes the $\sigma$-algebra generated by the family of $\F^o_{t}$-measurable $MZ$-continuous functions.
\end{corollary}
\begin{proof}
Let $\G^o_t$ denote the $\sigma$-algebra generated by the family of $\F^o_t$-measurable $MZ$-continuous functions. We have $\G^o_t\subset\F^o_t$ and $\F^o_{t-}\subset\G^o_t$ from Lemma~\ref{lem:lebesgue}. Moreover,  we have
\[
\omega^i(t)=\lim_{\varepsilon\rightarrow 0}\frac{1}{\varepsilon}\int_0^\varepsilon\omega^i(t+u)du,\ i\leq d,\ \varepsilon< T-t,
\]
where each $\omega\mapsto\frac{1}{\varepsilon}\int_0^\varepsilon\omega^i(t+u)du$ is an $MZ$-continuous function. Thus, the assertion follows.
\end{proof}
 
\begin{lemma}
\label{lem:lsc}
The mappings
\[
\omega\mapsto\|\omega\|_\infty\text{ and }\omega\mapsto N^{a,b}(\omega^i),\ a<b,\ i\leq d,
\]
are lower $MZ$-semicontinuous.
\end{lemma}
\begin{proof}
The proof is adapted from \cite{meyerzheng}. Let $i\leq d$ and $\omega^i_n\rightarrow_{MZ}\omega^i$ with $\sup_n\|\omega^i_n\|\leq c$. If $\|\omega^i\|_\infty>c$, then there exists $s<t$ such that $\omega^i(u)>c$, for all $u\in[s,t)$, or we have $\omega^i(T)>c$, either way, there exists an $MZ$-continuous function $F$ for which $\lim_{n\rightarrow\infty}F(\omega^i_n)< F(\omega^i)$, cf. \eqref{inmeas} and \eqref{interm}, so, this is a contradiction. Thus, the mapping $\omega\mapsto\|\omega\|_\infty:=\|\omega^1\|_\infty\vee\dots\vee\|\omega^d\|_\infty$ is lower $MZ$-semicontinuous. Similarly, one can show that the sets of the form $\{\omega:\exists u\in[s,t)\text{ s.t. }\omega^i(u)>b\}$ and $\{\omega:\exists u\in[s,t)\text{ s.t. }\omega^i(u)<a\}$, $s<t$, $a<b$ are open in the $MZ$-topology, from which the lower $MZ$-semicontinuity of the mappings $N^{a,b}$, $a<b$, follows. Indeed, let $a<b$ be fixed and consider a finite partition $\pi:=\{t_0<t_1<\cdots<t_n\}$ of $[0,t_n]$. We write $N^{a,b}_\pi(\omega^i)\geq k$, if one can find
\[
l_1<m_1\leq l_2<m_2\leq\cdots <l_k<m_k\leq n
\]
such that, for all $j<k$, $\omega^i(s)<a$, for some $s\in[t_{l_{j-1}},t_{l_j})$, and $\omega^i(t)>b$, for some $t\in[t_{m_{j-1}},t_{m_j})$ (or, for $t=T$, if $j=k$ and $m_k=n$). The partition $\pi$ is finite, so, the sets
\[
\{\omega:N^{a,b}_\pi(\omega^i)\geq k\}=\{\omega:N^{a,b}_\pi(\omega^i)>k-1\},\ k\in\mathbb{N},
\]
are open in the $MZ$-topology. Consequently, the mapping $\omega\mapsto N^{a,b}_\pi(\omega^i)$ and the mapping $N^{a,b}(\omega^i):=\sup_\pi N^{a,b}_\pi(\omega^i)$ are lower $MZ$-semicontinuous, for every $i\leq d$.
\end{proof}

We refer the reader to the book by Dellacherie and Meyer \cite[A,~IV]{dellacheriemeyer} and the paper \cite{meyerzheng} by Meyer and Zheng for details on pseudo-paths and the Meyer-Zheng topology, respectively.

\subsubsection{The Skorokhod's $J^1$-topology}
\label{sec:jone}

The following \emph{complete} metric, generating a topology called Skorokhod's $J^1$-topology, was introduced by Kolmogorov in \cite{kolmogorov2}. The metric introduced by Skorokhod himself in \cite{skorokhod} was not yet complete despite generating an equivalent topology.

\begin{definition}\textnormal{ 
The Skorokhod's $J^1$-topology on $\mathbb{D}([0,T];\mathbb{R}^d)$ is the topology generated by the complete metric
\begin{equation}
\label{skorometric}
J^1_T(\omega,\widetilde{\omega}):=\inf_{\lambda\in\Lambda}\left\{\sup_{s<t}\left|\log\frac{\lambda t-\lambda s}{t-s}\right|\vee\|\omega-\widetilde{\omega}\circ\lambda\|_\infty\right\},
\end{equation}
where $\Lambda$ denotes the class of strictly increasing, continuous mappings of $[0,T]$ onto itself and $i$ is the identity map on $[0,T]$.}
\end{definition}
\begin{definition}\textnormal{ 
The Skorokhod's $J^1$-topology on $\mathbb{D}([0,\infty);\mathbb{R}^d)$ is the topology generated by the (complete) metric
\begin{equation}
\label{eq:jinf}
J^1(\omega,\widetilde{\omega}):=\sum_{r=1}^\infty 2^{-r}(1\wedge J^1_r([\omega]^r,[\widetilde{\omega}]^r)),
\end{equation}
where $[\omega]^r$ indicates the restriction of $\omega$ on $[0,r]$.}
\end{definition}

Let us recall the criteria for the relative compactness in the $J^1$-topology from \cite{billingsley}. Let $\delta>0$ and denote
\[
m_\delta(\omega):=\inf_{t_i\in\pi}\max_{i\leq n}\sup_{s,t\in]t_{i-1},t_i]}|\omega(s)-\omega(t)|,\quad \omega\in\mathbb{D}([0,T];\mathbb{R}^d),
\]
where the infimum is take over all finite partitions $0=t_0<t_1<\cdots<t_n= T$ of $[0,T]$ with the mesh size $t_i-t_{i-1}>\delta$, for all $i\leq n$.
\begin{lemma}
A subset $K$ of $\mathbb{D}([0,T];\mathbb{R}^d)$, $T<\infty$, is relatively $J^1$-compact, if it is bounded and
\[
\lim_{\delta\rightarrow 0}\sup_{\omega\in K}m_\delta(\omega)=0.
\]
A subset $K$ of $\mathbb{D}([0,\infty);\mathbb{R}^d)$ is relatively $J^1$-compact, if the restriction of $K$ on $\mathbb{D}([0,T];\mathbb{R}^d)$ is relatively $J^1$-compact, for every $T<\infty$.
\end{lemma}
We refer the reader to \cite[Section~12,~16]{billingsley} for details on the Skorokhod's $J^1$-metric on $\mathbb{D}([0,T];\mathbb{R}^d)$ and $\mathbb{D}([0,\infty);\mathbb{R}^d)$, respectively.

\subsection{The proof of Lemma~\ref{lem:lowther}}
\label{sec:proof}

Recall that we claimed that there exists a uniform constant $b>0$ such that, for any $Q\in\P\left(\mathbb{D}(I;\mathbb{R}^d),\F_T\right)$, $H\in\E(Q)$ and $c>0$, we have
\begin{equation}
\label{eq:lowther}
Q[|(H\bullet X)_t|> c]\leq \frac{b}{c}\left[E_Q[|X_t|]+\sup_{H\in\E(Q)}E_Q[(H\bullet X)_t]\right],\ t\in I.
\end{equation}

\begin{proof}
The inequality \eqref{eq:lowther} is a generalizations of Burkholder's inequality, which states that there exists a uniform constant $a>0$ such that, for any $H\in\E(Q)$, $Q$-martingale $M$ and $c>0$, we have
\begin{equation}
\label{eq:burkholder}
Q[|H\bullet M|_t>c]\leq\frac{a}{c}E_Q[|M_t|],\ t\in I;
\end{equation}
see e.g. \cite{meyer} or \cite{bichteler} for a proof of \eqref{eq:burkholder}. For a fixed $Q\in\P\left(\mathbb{D}(I;\mathbb{R}^d),\F_T\right)$ and $H\in\E(Q)$, by \cite[B,~Appendix~2.3]{dellacheriemeyer}, we have
\begin{equation}
\label{eq:00}
\sup_{H\in\E(Q)}E_Q[(H\bullet X)_t]=\sum_{i=1}^d\text{Var}^Q_t(X^i),\ t\in I.
\end{equation}
Let us fix $i\leq d$ and assume that $E_Q[|X_t^i|]+\text{Var}^Q_t(X^i)$ is finite, otherwise, the result is trivial. Let $t^i_0<t^i_1<\cdots<t^i_n= t$ and $H=(H^i)_{i=1}^d$ be an element of $\E(Q)$, i.e.,
\[
H^i=\sum_{k=1}^nH^i_{t^i_k}\one_{]t^i_{k-1},t^i_k]},\ i\leq d,
\]
where each $|H^i_{t^i_k}|\leq 1$ is $\F_{t^i_k}$-measurable; cf. \eqref{eq:strats}. Consider the Doob decomposition
\[
X^i_{t^i_k}=M^i_{t^i_k}+A^i_{t^i_k},\ k=1,2,\dots,n,
\]
where $A^i_{t^i_k}=\sum_{j=1}^kE_Q[X^i_{t^i_j}-X^i_{t^i_{j-1}}\mid\F_{t^i_{j-1}}]$ and $M^i_{t^i_k}$ is a $Q$-martingale on $\{t^i_0,t^i_1,\dots,t^i_n\}$. We have
\begin{equation}
\label{eq:1}
Q[|(H^i\bullet A^i)_t|>c]\leq\frac{1}{c}E_Q[|(H^i\bullet A^i)_t|]\leq\frac{1}{c}\text{Var}^Q_t(X^i),\ i\leq d.
\end{equation}
Similarly, for $M^i$, we have
\[
E_Q[|M^i_t|]\leq E_Q[|X^i_t|+|A^i_t|]\leq E_Q[|X^i_t|]+\text{Var}^Q_t(X^i).
\]
Hence, by \eqref{eq:burkholder}, we have
\begin{equation}
\label{eq:2}
Q[|(H^i\bullet M^i)_t|>c]\leq\frac{a}{c}\left(E_Q[|X^i_t|]+\text{Var}^Q_t(X^i)\right),\ i\leq d.
\end{equation}
Combining \eqref{eq:00}, \eqref{eq:1} and \eqref{eq:2}, for $H\in\E(Q)$, $M=(M^i)_{i=1}^d$ and $A=(A^i)_{i=1}^d$, we get
\[
\begin{split}
Q[|(H\bullet X)_t|>c]&\leq Q[|(H\bullet M)_t+(H\bullet A)_t|>c]\\
&\leq \sum_{i=1}^d Q\left[|(H^i\bullet M^i)_t|>\frac{c}{2d}\right]+\sum_{i=1}^d Q\left[|(H^i\bullet A^i)_t|>\frac{c}{2d}\right]\\
&\leq \frac{2ad}{c}\sum_{i=1}^d\left(E_Q[|X^i_t|]+\text{Var}^Q_t(X^i) \right)+\frac{2d}{c}\sum_{i=1}^d\text{Var}^Q_t(X^i)\\
&\leq \frac{b}{c}\left(E_Q[|X_t|]+\sup_{H\in\E(Q)}E_Q[(H\bullet X)_t] \right),\ t\in I,\ c>0,
\end{split}
\]
where $b:=2(a+1)d$. The proof is completely similar for the filtration $(\F^o_t)_{t\in I}$.
\end{proof}

\bibliographystyle{alpha}
\bibliography{sp}

\end{document}